\documentclass[11pt,reqno]{amsart}
\usepackage[T1]{fontenc}
\usepackage[utf8]{inputenc}
\usepackage{lmodern}
\usepackage{amsmath,amssymb,amsthm,mathtools,mathrsfs}
\usepackage[margin=1in]{geometry}
\usepackage{microtype,booktabs,array}
\usepackage[numbers,sort&compress]{natbib}
\usepackage[colorlinks=true,linkcolor=blue,citecolor=blue,urlcolor=blue]{hyperref}
\usepackage{zref-clever}
\zcsetup{cap=true,abbrev=false}
\usepackage{aliascnt}
\numberwithin{equation}{section}
\newtheorem{theorem}{Theorem}[section]
\newaliascnt{lemma}{theorem}
\newtheorem{lemma}[lemma]{Lemma}
\aliascntresetthe{lemma}
\newaliascnt{proposition}{theorem}
\newtheorem{proposition}[proposition]{Proposition}
\aliascntresetthe{proposition}
\newaliascnt{corollary}{theorem}

\aliascntresetthe{corollary}
\theoremstyle{definition}
\newaliascnt{definition}{theorem}

\aliascntresetthe{definition}
\theoremstyle{remark}
\newaliascnt{remark}{theorem}

\aliascntresetthe{remark}
\newcommand{\E}{\mathbb E}
\newcommand{\R}{\mathbb R}

\newcommand{\dd}{\,\mathrm d}
\DeclareMathOperator{\Var}{Var}
\DeclareMathOperator{\Cov}{Cov}
\DeclareMathOperator{\supp}{supp}
\DeclareMathOperator{\sech}{sech}

\allowdisplaybreaks[1]
\title[Parisi measure of the S--K model with external field]{Parisi measure of the Sherrington--Kirkpatrick model with external field}
\author{Patrick Lopatto}
\hypersetup{pdftitle={Parisi measure of the Sherrington--Kirkpatrick model with external field},pdfauthor={Patrick Lopatto}}
\date{\today}
\begin{document}
\begin{abstract}
We prove that the Parisi measure of the Sherrington--Kirkpatrick model has
connected support at every positive temperature and every 
positive external field. When the model breaks replica symmetry, its
support is an interval contained in $(0,1)$, its restriction to the interior
has a smooth density that remains strictly positive up to both endpoints,
and both endpoints carry atoms. The proof rests on a strict inequality
for moments of spatial derivatives of the Parisi PDE, which both excludes
gaps in the support and forces positivity of the density. The same
estimates recover the zero-field support description and give a lower
bound for the mass of its upper atom.
\end{abstract}
\maketitle
\tableofcontents

\section{Introduction}\label{sec:introduction}

The Parisi measure is the order parameter of a mean-field spin glass, introduced in Parisi's description of replica symmetry breaking \cite{Parisi1979,Parisi1980OrderParameter}. Its support encodes the hierarchy of overlap scales: a point mass represents replica symmetry, while a richer support describes progressively more elaborate forms of symmetry breaking. A basic question is whether replica symmetry breaking in the Sherrington--Kirkpatrick model \cite{SherringtonKirkpatrick1975} produces a genuinely continuous hierarchy of states. Understanding the geometry of this support is therefore central to describing the form of symmetry breaking predicted by the Parisi theory.

The foundation of the mathematical theory of mean-field spin glasses is the 
Parisi variational formula for the limiting free energy. After Guerra and
Toninelli established the thermodynamic limit
\cite{GuerraToninelli2002Thermodynamic}, Guerra proved the sharp
interpolation upper bound \cite{Guerra2003}, while Aizenman, Sims, and
Starr developed an extended variational principle for the
Sherrington--Kirkpatrick free energy \cite{AizenmanSimsStarr2003}.
Talagrand proved the matching lower bound and hence the Parisi formula
\cite{Talagrand2006}, and Panchenko later extended the formula to general
mixed \(p\)-spin models \cite{Panchenko2014Parisi}. Talagrand also identified the Parisi measure with the limiting overlap
distribution for mixed even $p$-spin models at positive external field,
under genericity assumptions on the interaction coefficients
\cite{Talagrand2003OrderParameter}. 
Complementing this variational theory, Panchenko
\cite{Panchenko2013Ultrametricity} proved that limiting overlap arrays
satisfying the Ghirlanda--Guerra identities
\cite{GhirlandaGuerra1998} are ultrametric. Auffinger and Chen
\cite{AuffingerChen2015Unique} proved strict convexity of the Parisi
functional and uniqueness of its minimizing measure. These results identify
the free energy and its order parameter, but leave open the structure of the
minimizer for a given inverse temperature and external field.

The analysis of the minimizer has developed in parallel with probabilistic and PDE representations of the Parisi functional. Panchenko \cite{Panchenko2005ParisiFunctional} obtained an early convexity result for the functional, and W.-K.\ Chen \cite{Chen2015ConvexityPDE} developed a PDE approach to this problem. Auffinger and W.-K.\ Chen \cite{AuffingerChen2015Unique} introduced a stochastic variational representation of the Parisi PDE, and Jagannath and Tobasco \cite{JagannathTobasco2016Dynamic} gave a dynamic-programming formulation and an alternative derivation of its basic properties. 
Talagrand \cite{Talagrand2006ParisiMeasures} studied differentiability of
the Parisi functional, and Panchenko \cite{Panchenko2008Differentiability}
gave an alternative approach to differentiability of the Parisi formula. Auffinger and W.-K.\ Chen \cite{AuffingerChen2015Properties} developed
regularity estimates for the Parisi PDE and corresponding self-consistency
conditions at support points. W.-K.\ Chen \cite{Chen2017Variational}, and
independently Jagannath and Tobasco \cite{JagannathTobasco2017}, subsequently
obtained first-order variational characterizations of Parisi measures. In these characterizations, the support is contained in the set of global minima of a function determined by the solution of the Parisi PDE.

At positive field, the first question is whether the measure is a point
mass. The stability calculation of de Almeida and Thouless
\cite{deAlmeidaThouless1978} predicts the boundary of the replica-symmetric
region in the Sherrington--Kirkpatrick model. Rigorous replica-symmetric regions at high temperature and large
external field were established by Talagrand \cite{Talagrand2000RSB} and
Guerra and Toninelli \cite{GuerraToninelli2002Quadratic}, while Toninelli
\cite{Toninelli2002} proved failure of replica symmetry when the
de Almeida--Thouless stability condition is violated. Jagannath and Tobasco
\cite{JagannathTobasco2017} showed that the predicted boundary agrees with
the phase boundary outside a bounded set of parameters. Bolthausen
\cite{Bolthausen2019Morita} proved the replica-symmetric free-energy formula
for sufficiently small inverse temperature, and Brennecke and Yau
\cite{BrenneckeYau2022} extended its range of validity. Replica symmetry
throughout the region up to the de Almeida--Thouless line was recently proved in
\cite{Lopatto2026AT}. This determines when the minimizing measure is a
point mass, but not its support or singular components when it is
nontrivial.

At zero field, Auffinger and W.-K.\ Chen
\cite{AuffingerChen2015Properties} proved that zero belongs to the support,
that the distribution function is smooth on intervals contained in the
support, and that neither one nor two support points suffice at sufficiently
low temperature. Auffinger, W.-K.\ Chen, and Zeng
\cite{AuffingerChenZeng2020} proved that the zero-temperature Parisi measure
has infinitely many support points. Their finite-temperature consequence
implies, in particular, that at zero field the number of support points
exceeds every fixed bound once the temperature is sufficiently low. 

For fixed positive temperatures at zero field, Zhou \cite{Zhou2026FRSB}
proved interval support for inverse temperatures slightly above the critical
value. A recent preprint \cite{Lopatto2026} established the support description
throughout the low-temperature phase: the support is an interval beginning
at zero, the measure has a smooth density below its largest support point,
and its only atom is at that point. H.-B.\ Chen
\cite{Chen2026EndpointMass} obtained a uniform lower bound for this atom
and quantitative estimates as temperature tends to zero, together with
full-interval support for the zero-temperature measure. In the related
Ising pure \(p\)-spin models with \(p\ge 3\) and no external field, Zhou
\cite{Zhou2024Gardner,Zhou2026GardnerII} recently established the two
Gardner transitions, from replica symmetry to one-step replica symmetry
breaking and then to full replica symmetry breaking. These results give a
detailed picture of symmetry breaking at zero field, but leave open the
corresponding structure of the Parisi measure in the SK model at positive
external field.

We address this problem by proving connected support for every
$0<\beta<\infty$ and $h>0$. Whenever the measure
is not replica symmetric, its support is a nondegenerate interval contained
in \((0,1)\), and its restriction to the interior has a smooth density that
remains strictly positive up to both endpoints. Both endpoints carry atoms,
and there is no other singular component. Thus the result describes the
entire measure, not only the number of overlap levels. In contrast to the
zero-field case, the lower endpoint is strictly positive and also carries
positive mass.

\subsection{Main results}

Fix an inverse temperature $\beta>0$ and a field $h\ge0$. For an integer
$N\ge1$, let $(g_{ij})_{1\le i<j\le N}$ be independent standard Gaussian
variables.
The Sherrington--Kirkpatrick Hamiltonian is
\[
 H_N(\sigma)=\frac1{\sqrt N}\sum_{1\le i<j\le N}g_{ij}\sigma_i\sigma_j,
 \qquad \sigma\in\{-1,1\}^N.
\]
We use Gibbs weight $\exp(\beta H_N(\sigma)+h\sum_i\sigma_i)$, so $h$ is
the field in the Gibbs exponent. Write $\nu_{\beta,h}$ for the unique minimizer over probability measures
on $[0,1]$ of the Parisi functional displayed in
\eqref{eq:parisi-functional-time} below, recalling that uniqueness was proved in \cite{AuffingerChen2015Unique}. The minimizer is called replica
symmetric when it is a point mass.

\begin{samepage}
\begin{theorem}\label{thm:main}
For every $0<\beta<\infty$ and $h>0$, there exist
$0<q_-\le q_+<1$ such that
\[
 \supp\nu_{\beta,h}=[q_-,q_+].
\]
The interval is nondegenerate if and only if the minimizer is not
replica symmetric. If $q_-<q_+$, then
\[
 \nu_{\beta,h}=a_-\delta_{q_-}
       +f(q)\mathbf1_{(q_-,q_+)}\dd q+a_+\delta_{q_+},
 \qquad a_->0,\qquad a_+>\frac38.
\]
The function $f$ is smooth and strictly positive on $[q_-,q_+]$,
where smoothness at the endpoints means that every derivative has a
continuous one-sided extension from the interior.
\end{theorem}
\end{samepage}
The proof is given in \zcref{sec:structure}. 
The same methods also give another proof of the zero-field support description shown in a recent preprint \cite{Lopatto2026}, as discussed in \zcref{app:zero-field}. 

\subsection{Ideas of the proof}

The main challenge of the proof is finding the right observables and inequalities to extend the zero-field approach of \cite{Lopatto2026} to positive external field. These were located using a massive number of numerical experiments orchestrated by large language models. Due to this approach to proof discovery, the resulting argument is heavy with computation; this is most evident in the fourth step of the outline below.

We will prove an inequality that yields both connected support and positivity of
the density. On a proposed gap, it forces a third derivative to be positive
where the variational conditions force it to be nonpositive. On an interval
of support, it instead forces the distribution function to have positive
derivative. 

Write $L=\beta^2$ and $m(t)=\nu_{\beta,h}([0,t/L])$. The variable $t=Lq$ is
a rescaled overlap coordinate, used as the time parameter of an auxiliary
diffusion. The coefficient $m$ determines the
Parisi PDE solution $U$ and the diffusion $Y$ through
\[
 \begin{aligned}
 U_t&=-\frac12(U_{xx}+m(t)U_x^2),& U(L,x)&=\log\cosh x,\\
 \dd Y_t&=m(t)U_x(t,Y_t)\dd t+\dd W_t,&Y_0&=h,
 \end{aligned}
\]
where $W$ is standard Brownian motion. Thus the external-field parameter
$h$ in the Gibbs weight is the starting point of $Y$. Put $B=U_x$ and
$C=U_{xx}$. In expectations below, spatial derivatives are evaluated at
$(t,Y_t)$ unless another argument is specified.

\medskip
\emph{Step 1: Reduce a gap to strict positivity of $Q'''$.}
Define
\[
 Q(t)=\E B(t,Y_t)^2,\qquad S(t)=\E C(t,Y_t)^2=Q'(t).
\]
A gap in the support makes the cumulative distribution function $m$
constant on the corresponding time interval. If its endpoints are
$t_-,t_+$, the variational conditions stated below give
$Q(t_-)/t_-=Q(t_+)/t_+=1/L$ and $S(t_-),S(t_+)\le1/L$.
Define
\[
 N(t)=Q(t)-tS(t).
\]
Since $Q'=S$,
\[
 \left(\frac{Q(t)}t\right)'
 =\frac{tS(t)-Q(t)}{t^2}
 =-\frac{N(t)}{t^2}.
\]
The endpoint conditions on $S$ imply
$N(t_-),N(t_+)\ge0$, while the equality of $Q(t)/t$ at the two
endpoints gives
\[
 \int_{t_-}^{t_+}\frac{N(t)}{t^2}\dd t=0.
\]
Unless $N$ vanishes identically on the gap, it must be negative
somewhere in its interior. Since it is nonnegative at both endpoints,
it  attains a negative interior minimum.

On the gap, using that $m$ is constant gives the differentiation formulas
\[
 Q''=\E(C_x^2-2mC^3),\qquad
 Q'''=J(t;m)=\E(C_{xx}^2-12mCC_x^2+6m^2C^4).
\]
At an interior minimum of $N$, the identities $N'=-tQ''$ and
$N''=-Q''-tQ'''$ give $Q''=0$ and $Q'''\le0$. It is thus enough to prove
\[
 N(t)\le0,\qquad Q''(t)=0\quad\Longrightarrow\quad J(t;m)>0.
\]
This implication is a special case of \zcref{thm:derivative-moments}.  The remaining steps establish the strict inequality
and explain how it also determines the density on the support.
The variational conditions and differentiation formulas are proved in
\zcref{sec:variation}.

\medskip
\emph{Step 2: Normalize the spatial derivatives and control the Gaussian parameters.}
Fix a positive time $t$. We divide $C$ by its value at the origin and choose
a spatial scale so that the normalized function $A$ satisfies $A(0)=1$
and its logarithmic derivative $u=-A'/(2A)$ satisfies $u'(0)=1$. Explicitly,
\[
 c_0=C(t,0),\qquad \lambda^2=-\frac{C_{xx}(t,0)}{2c_0},\qquad
 A(v)=\frac{C(t,v/\lambda)}{c_0},\qquad u(v)=-\frac{A'(v)}{2A(v)}.
\]
Here $v=\lambda x$; primes on $A$ and $u$ denote derivatives in $v$.
Our goal is to show that, if $N(t)\le0$, then the Gaussian comparison
below cannot have too large a tilt parameter. We normalize the spatial
profile so that this restriction can be expressed using two dimensionless
parameters.
The estimates in \zcref{sec:profiles} give $\lambda>0$ and, for $v\ge0$,
\[
 0<u'\le1,\qquad \tanh v\le u(v)\le v,\qquad vu'(v)\le u(v),
 \qquad e^{-v^2}\le A(v)\le\sech^2v.
\]
The normalization puts solutions for different coefficients $m$ on a common
scale. The inequality $vu'\le u$ is preserved by the interpolation used
below, while the fourth inequality in the previous display 
places $A$ between the Gaussian profile and the terminal profile of the PDE.

We now make two comparisons. The first replaces the law of $Y_t$ by an
explicit Gaussian law while keeping the spatial functions $B$ and $C$
fixed. The drift of $|Y|$ points away from zero, so comparison with
reflected Brownian motion and monotonicity of
$B(t,x)^2-tC(t,x)^2$ for $x\ge0$ give
\[
 N(t)\ge\E\bigl[B(t,h+\sqrt t Z)^2-tC(t,h+\sqrt t Z)^2\bigr],
 \qquad Z\sim\mathcal N(0,1).
\]
Thus $N(t)\le0$ forces the expectation on the right to be nonpositive.
We will show that it becomes strictly positive once $\alpha$ is large enough. 
For the corresponding normalized Gaussian, set
\[
 \alpha=\frac h{\lambda t},\qquad d=\frac1{\lambda^2t}.
\]
Then
\[
 \lambda(h+\sqrt t Z)\sim\mathcal N\left(\frac\alpha d,\frac1d\right).
\] 
The density of the absolute value of this Gaussian is proportional to
$e^{-dv^2/2}\cosh(\alpha v)$ on $v\ge0$. Thus $d$ is its reciprocal variance, while $\alpha$ is the coefficient of
$v$ in its exponential density. We refer to $\alpha$ as the tilt parameter.

This Gaussian expectation still involves the unknown function $A$ and its
primitive $I(v)=\int_0^v A(y)\dd y$. As we show below, it has the opposite sign to
\[
 \mathcal E_\alpha[A]=\int_0^\infty
       (A(v)^2-dI(v)^2)e^{-dv^2/2}\cosh(\alpha v)\dd v.
\]
The second comparison replaces the unknown profile $A$ by the explicit
Gaussian profile $A_G(v)=e^{-v^2}$ while preserving the sign information
needed above. We interpolate between them by
\[
 A_\varepsilon=A^{1-\varepsilon}e^{-\varepsilon v^2},
 \qquad 0\le\varepsilon\le1.
\]
For $0<\alpha\le1$, we prove that every zero of
$\varepsilon\mapsto\mathcal E_\alpha[A_\varepsilon]$ is crossed in the
increasing direction:
\[
 \mathcal E_\alpha[A_\varepsilon]=0
 \quad\Longrightarrow\quad
 \frac{\dd}{\dd\varepsilon}\mathcal E_\alpha[A_\varepsilon]>0.
\]
Therefore a negative value at the Gaussian endpoint
$A_1=A_G$ forces a negative value at the original profile $A_0=A$.
The proof uses the preserved inequality $vu'\le u$ to show that the
derivative kernel changes sign exactly once; see
\zcref{sec:gaussian-comparison}.

For an explicit absolute constant $\kappa\in(0,1)$ specified in
\zcref{sec:gaussian-comparison}, define the threshold 
$\alpha_*(d)>0$ by
\[
 \alpha_*(d)^2=\min\{\kappa,\sqrt{\kappa d}\}.
\]
The Gaussian calculation gives
$\mathcal E_{\alpha_*(d)}[A_G]<0$, and the interpolation therefore gives
$\mathcal E_{\alpha_*(d)}[A]<0$. Since the Gaussian expectation above
has the opposite sign from $\mathcal E_{\alpha}[A]$, it is strictly
positive at $\alpha=\alpha_*(d)$. It is strictly increasing in $\alpha$,
so it remains positive for every $\alpha\ge\alpha_*(d)$.

But $N(t)\le0$ forces this Gaussian expectation to be nonpositive.
Hence $\alpha<\alpha_*(d)$.
Using the definition of $\alpha_*(d)$, this is equivalent to
\[
 N(t)\le0\quad\Longrightarrow\quad
 \theta=\alpha^2<\kappa,\qquad d>\frac{\theta^2}{\kappa}.
\]
These are conditional bounds on rescaled
combinations of $h$ and $t$ at the time under consideration; $h>0$ itself
is arbitrary. Their role is to control the terms introduced by the nonzero
starting point of $Y_t$ in the later integration-by-parts estimates.

\medskip
\emph{Step 3: Estimate the weighted diffusion density.}
After expressing $C_x$ and $C_{xx}$ in the normalized variables,
the integrand defining $J(t;m)$ has a common factor $C(t,Y_t)^2$.
We absorb
it into a probability measure by setting
\[
 \int f(v)\rho(\dd v)=
 \frac{\E[C(t,Y_t)^2f(\lambda|Y_t|)]}{S(t)}.
\]
For the constant value $m$ in the gap, put
$a(v)=mC(t,v/\lambda)/\lambda^2$. The condition $Q''(t)=0$ arising at an interior minimum of $N$
has a simple form in these normalized variables:
\[
 Q''(t)=0\quad\Longleftrightarrow\quad
 \E_\rho a=2\E_\rho u^2.
\] 
At a positive time the left side cannot hold with $m=0$, since
$\E C_x(t,Y_t)^2>0$. We may therefore work with $m>0$. Both this moment
identity and the expression whose positivity gives $J>0$ are now averages
under the same measure $\rho$. The remaining problem is to control these moments under $\rho$. For this we need quantitative information about
the density of $\rho$. Its logarithmic derivative will enter the
integration-by-parts identities in the next step, while concavity of the
non-Gaussian correction will give the sign needed to bound the resulting
variance terms.

To estimate its density, let $p_h(\sigma,x)$ be the density of $Y_\sigma$.
Multiplication of the absolute-value density by $e^{-mU}$ removes the
spatial drift from its forward equation. For smooth nondecreasing $m$,
the transformed density satisfies
\[
 g(\sigma,x)=e^{-m(\sigma)U(\sigma,x)}
              [p_h(\sigma,x)+p_h(\sigma,-x)],\qquad
 g_\sigma=\frac12g_{xx}-m'(\sigma)Ug.
\]
We compare $g$ with the Neumann heat kernel
$k_h(\sigma,x)=K_\sigma(x-h)+K_\sigma(x+h)$, where
$K_\sigma(x)=(2\pi\sigma)^{-1/2}e^{-x^2/(2\sigma)}$.
On the half-line, this is the density of $|h+\sqrt\sigma Z|$.
Write $r=x^2$. Since $m'\ge0$, monotonicity of
$U(\sigma,\sqrt r)$ gives a nonnegative source in the equation for the
first $r$-derivative of $-\log(g/k_h)$. Once this derivative is
nonnegative, concavity of $U(\sigma,\sqrt r)$ and the explicit Neumann
kernel give nonpositive sources in the equation for its second
$r$-derivative. Parabolic comparison propagates these signs. Hence the
negative logarithm of $g/k_h$ is nondecreasing and concave in $x^2$.
The Brownian-bridge representation in \zcref{sec:diffusion-density}
justifies the derivative estimates and passage to general nondecreasing
coefficients.

Restoring the explicit factors and applying the $C^2$ weight gives
\[
 \begin{gathered}
 \rho(\dd v)=\mathcal Z^{-1}\omega(v)
       e^{-dv^2/2-\psi(v^2)}\cosh(\alpha v)\dd v,\\
 \omega(v)=C(t,v/\lambda)^2e^{mU(t,v/\lambda)},\qquad
 \psi'\ge0,\quad\psi''\le0.
 \end{gathered}
\]
Here $\mathcal Z$ is the normalizing constant, and $\psi$ is the negative
logarithm of the preceding ratio after spatial rescaling. Primes on
$\psi$ denote derivatives with respect to its argument $r$. The factor
$\cosh(\alpha v)$ displays the Gaussian contribution of the starting point;
$\psi$ may still depend on $h$. The derivative estimates and concavity give
a bounded coefficient $s(v)=d+2\psi'(v^2)\ge d$ with $s'(v)\le0$ for
$v\ge0$.

\medskip
\emph{Step 4: Combine the moment inequalities to prove $J>0$.}
Write
\[
 w=u^2,\qquad j=u'-a,\qquad
 \mathcal T=1-a+2w-j,\qquad M=\E_\rho a.
\]
The moment identity from Step~3 gives $\E_\rho w=M/2$. After factoring
a positive normalization, the inequality $J>0$ reduces to positivity of
a dimensionless expression $O$ involving the variances of $a$ and
$\mathcal T$, the covariance of $a$ and $w$, and the mean of $j$.

The density estimate from Step~3 gives integration-by-parts inequalities
under $\rho$. Combined with the differential inequalities for $a,j,u$,
these control the terms in $O$. We organize the remaining argument according
to
\[
 0<\mu=a(0)=\frac{mc_0}{\lambda^2}\le1.
\]
We cover this range by three cases. 
For
$0<\mu\le11/20$, two variance inequalities already give a strictly
positive lower bound for $O$. For $\mu\ge29/40$, the
integration-by-parts estimates give incompatible upper and lower bounds
on the normalized mean of $j$, so the moment identity cannot hold.

It remains to treat
\[
 \frac{11}{20}\le\mu\le\frac{29}{40}.
\]
In this range we combine five explicitly nonnegative quantities obtained
from variances, covariance inequalities, and the pointwise differential
bounds. With suitable positive coefficients, their sum reduces the desired
estimate to
\[
 O\ge
 \mathfrak a\theta^2+\mathfrak b\theta+\mathfrak c,
 \qquad \mathfrak a>0.
\]
The exact algebra in \zcref{app:arithmetic} verifies
\[
 4\mathfrak a\mathfrak c-\mathfrak b^2>0
\]
throughout this intermediate range. Hence $O>0$ and therefore
$J(t;m)>0$. Sections~\ref{sec:weighted-moments}--\ref{sec:positivity}
give the detailed moment inequalities and algebraic combination.

\medskip
\emph{Step 5: Recover the density and endpoint atoms.}
Once gaps have been excluded, suppose that the support is the
nondegenerate interval $[q_-,q_+]$, and set
\[
 t_-=Lq_-,\qquad t_+=Lq_+.
\]
Every point of this interval is a support point, so the variational
characterization in \zcref{sec:variation} gives
\[
 Q(t)=\frac{t}{L}\qquad (t_-\le t\le t_+).
\]
In particular $Q''=0$ in the interior. Since
\[
 Q''=F-2mG,\qquad
 F(t)=\E C_x(t,Y_t)^2,\qquad
 G(t)=\E C(t,Y_t)^3>0,
\]
we obtain
\[
 m(t)=\frac{F(t)}{2G(t)}
 \qquad (t_-\le t<t_+).
\]
This identity already gives regularity of the measure. The functions
$F$ and $G$ are Lipschitz, and $G$ is bounded away from zero on
$[t_-,t_+]$, so the quotient $F/(2G)$ is Lipschitz. Hence the
distribution function $m$ is Lipschitz in the interior of the support:
there are no interior atoms or singular continuous mass. Adapting the
induction in \cite[Theorem~2(ii)]{AuffingerChen2015Properties} and iterating
the differentiated moment identities then gives smoothness, with one-sided
derivatives at $t_-$ and $t_+$.

It remains to show that this density is strictly positive. Differentiating
$F=2mG$ gives
\[
 2G(t)m'(t)
 =F'(t)-2m(t)G'(t)
 =J(t;m(t)),
\]
because $Q'''=0$ on the support. Moreover,
\[
 N(t)=Q(t)-tQ'(t)=0
\]
there, and $F-2mG=0$. Thus \zcref{thm:derivative-moments} applies and
gives
\[
 J(t;m(t))>0.
\]
Since $G(t)>0$, it follows that
\[
 m'(t)>0,
\]
and therefore the density in the original overlap coordinate,
\[
 f(q)=Lm'(Lq),
\]
is strictly positive. Thus the same inequality that ruled out gaps in
Step~1 now shows that the distribution function increases strictly
throughout its support.

We also need the corresponding one-sided statements at the endpoints.
For this purpose, \zcref{thm:derivative-moments} allows the coefficient
appearing in $F-2m_0G$ and $J(t;m_0)$ to be a fixed value $m_0$, while
leaving the PDE solution and diffusion unchanged, provided
\[
 m\le m_0\quad\text{before }t,\qquad
 m\ge m_0\quad\text{after }t.
\]
At $t_-$ and $t_+$ we take $m_0$ to be the value approached from the
interior of the support. This gives $J>0$ at both endpoints and hence
strictly positive one-sided limits of the density.

Finally, these endpoint values also identify the atoms. Since $t_->0$,
the diffusion law has positive density and $C_x(t_-,\cdot)$ is not
identically zero, so
\[
 F(t_-)>0.
\]
The quotient formula therefore gives $m(t_-)>0$. There is no mass below
$q_-$, so
\[
 \nu_{\beta,h}(\{q_-\})=m(t_-)>0.
\]
At the upper endpoint, write
\[
 m_+=m(t_+-).
\]
Because $m(t)=1$ for $t>t_+$, the Parisi PDE can be solved explicitly
above $t_+$, giving
\[
 U(t_+,x)=\log\cosh x+\frac{L-t_+}{2}.
\]
Substituting this terminal profile into the moment bounds reduces the
problem to a scalar inequality for $m_+$, from which we obtain
$m_+<5/8$. Hence
\[
 \nu_{\beta,h}(\{q_+\})=1-m_+>\frac38.
\]
The details of the regularity and endpoint calculations are given in
\zcref{sec:structure}.

\subsection{Acknowledgments}

P.L. was partially supported by NSF grant DMS-2450004. This paper was written
by the author with assistance from large language models. That assistance
included suggestions for arguments, help with drafting and revision, and
code for computational checks and experiments.

\section{The variational reduction}\label{sec:variation}

This section states the variational conditions and moment identities in the rescaled
time coordinate used to exclude support gaps. Brownian motions are defined
with the usual augmentation of their natural filtrations.

Fix $\beta>0$ and $h\ge0$, and let $\nu$ be a Borel probability measure
on $[0,1]$. Put $L=\beta^2$ and $m(t)=\nu([0,t/L])$ for $0\le t\le L$.
We use the solution of
\begin{equation}\label{eq:parisi-time}
 U_t=-\frac12(U_{xx}+m(t)U_x^2),\qquad
 U(L,x)=\log\cosh x,
\end{equation}
and the diffusion
\begin{equation}\label{eq:optimal-diffusion}
 \dd Y_t=m(t)U_x(t,Y_t)\dd t+\dd W_t,\qquad Y_0=h.
\end{equation}
The regularity facts used below are proved in \zcref{lem:pde-regularity}.
In particular, the drift is bounded and globally Lipschitz in space.
The PDE and diffusion depend on $m$ only through its almost-everywhere values.

Write $B=U_x$, $C=U_{xx}$, and set
\begin{equation}\label{eq:moment-QSFG}
 Q(t)=\E B(t,Y_t)^2,\quad S(t)=\E C(t,Y_t)^2,\quad
 F(t)=\E C_x(t,Y_t)^2,\quad G(t)=\E C(t,Y_t)^3.
\end{equation}
Constants in estimates for the spatial derivatives may depend on $L$
and the fixed derivative order, but are uniform over coefficients taking
values in $[0,1]$. The functions $F,G,S$ are Lipschitz on $[0,L]$, and $G>0$.
For a fixed value $m_0$, expressions such as $J(t;m_0)$ below always use
the PDE solution and diffusion associated with the original coefficient $m$.

The following identities are the stochastic differentiation formulas of
\cite[Lemma~2]{AuffingerChen2015Unique} and
\cite[Proposition~3]{AuffingerChen2015Properties}, together with their
immediate higher-order consequences. The SK third-derivative formula is
also recorded in \cite[Lemma~2.5]{Lopatto2026}. We include the calculation
in the time coordinate $t=\beta^2q$.

\begin{lemma}\label{lem:ito-moments}
For every nondecreasing measurable $m:[0,L]\to[0,1]$,
\begin{align}
 Q'&=S,\label{eq:ito-Q}\\
 S'&=F-2mG,\label{eq:ito-S}\\
 F'&=\E(C_{xx}^2-6mCC_x^2),\label{eq:ito-F}\\
 G'&=3\E(CC_x^2-mC^4)\label{eq:ito-G}
\end{align}
in integral form. More precisely,
\[
 Q\in C^1([0,L]),\qquad S,F,G\in W^{1,\infty}([0,L]),
\]
and \eqref{eq:ito-Q} holds pointwise in the interior, with one-sided
interpretation at the endpoints, while \eqref{eq:ito-S}--\eqref{eq:ito-G}
hold almost everywhere. 
If $m=m_0$ almost everywhere on an open interval $I\subset(0,L)$, then
$Q$ is smooth on $I$ and
\begin{equation}\label{eq:ito-J}
 Q'''(t)=J(t;m_0)
 =\E(C_{xx}^2-12m_0CC_x^2+6m_0^2C^4),
 \qquad t\in I.
\end{equation}
More generally, put $D_j=\partial_x^jU$, $j\ge1$. For every integer
$k\ge1$ and every polynomial $P$ in $k$ variables, the moment
\[
 M_P(t)=\E\bigl[P(D_1(t,Y_t),\ldots,D_k(t,Y_t))\bigr]
\]
is Lipschitz on $[0,L]$. There are polynomials $P_0,P_1$ in at most
$k+1$ variables, determined by It\^o's formula and
\eqref{eq:all-derivative-ito}, such that
\[
 M_P'=M_{P_0}+mM_{P_1}
\]
almost everywhere.
\end{lemma}

\begin{proof}
For $\mathscr L=\partial_t+(1/2)\partial_{xx}+mB\partial_x$,
differentiating \eqref{eq:parisi-time} in space gives
\[
 \mathscr L B=0,\qquad \mathscr L C=-mC^2,
 \qquad \mathscr L C_x=-3mCC_x.
\]
Suppose first that $m$ is a step function with finitely many jumps. On each
open interval between successive jumps, the differentiated integral
PDE in \zcref{lem:pde-regularity} has continuous integrands. Thus the
spatial derivatives of $U$ are $C^1$ in time and smooth in space there,
which gives the $C^{1,2}$ regularity required for classical It\^o's
formula on compact subintervals.  The spatial-derivative bounds make all
stochastic integrals below square-integrable martingales. Passing to the
interval endpoints and summing gives the integral identities on $[0,L]$,
since the spatial derivatives of $U$ are continuous across the partition
times. Applying It\^o's formula to $B^2,C^2,C_x^2,C^3$ then gives
\eqref{eq:ito-Q}--\eqref{eq:ito-G} for step functions.

To obtain integral identities for polynomial moments of the spatial
derivatives, set $D_j=\partial_x^jU$, $j\ge1$. Differentiating the
Parisi PDE and applying It\^o's formula along $Y$ gives
\begin{equation}\label{eq:all-derivative-ito}
 \begin{aligned}
 \dd D_j(t,Y_t)={}&-\frac{m(t)}2
 \sum_{\ell=1}^{j-1}\binom j\ell
 D_{\ell+1}(t,Y_t)D_{j-\ell+1}(t,Y_t)\dd t\\
 &+D_{j+1}(t,Y_t)\dd W_t,
 \end{aligned}
\end{equation}
where the sum is empty for $j=1$. For a polynomial $P$ in
$D_1,\ldots,D_k$, the first-derivative terms in It\^o's formula
contain one explicit factor of $m(t)$, while the quadratic-covariation
terms contain products $D_{i+1}D_{j+1}$ and no explicit factor of $m(t)$.
Thus the drift has the form $P_0+m(t)P_1$, where $P_0,P_1$ are
polynomials in $D_1,\ldots,D_{k+1}$. The bounds in
\zcref{lem:pde-regularity} make the stochastic integrals square-integrable
martingales and justify the corresponding expectation identities in
integral form.

Approximate a general nondecreasing $m$ in $L^1$ by nondecreasing step
functions. The bounds on spatial derivatives and their stability imply
uniform convergence of the spatial functions. Drive the approximating
diffusions $Y^n$ and $Y$ by the same Brownian motion. Since $|B|\le1$ and
$|B_x|\le1$, Gronwall's inequality gives
\[
 \sup_{t\le L}|Y_t^n-Y_t|
 \le e^L\left(\|m_n-m\|_{L^1[0,L]}
       +L\sup_{t,x}|B_n(t,x)-B(t,x)|\right)\longrightarrow0.
\]
The moment integrands are bounded, so dominated convergence proves their
integral identities, including at jump times. The same approximation applies to every fixed polynomial in the $D_j$.
The stability estimates for the spatial derivatives, together with
$Y^n\to Y$ uniformly, give convergence of the corresponding polynomial
evaluations along the diffusions. Their drift terms converge in $L^1$,
while the stochastic-integral terms converge in $L^2$ by the It\^o
isometry. Passing to the limit in the integral semimartingale identities
therefore gives \eqref{eq:all-derivative-ito} and the corresponding
polynomial-moment identities for general $m$. 
The continuity of the spatial derivatives in time and space, together
with continuity of $Y$ and their uniform bounds, implies that $S$ is
continuous. Hence the integral identity for $Q$ gives
$Q\in C^1([0,L])$ with $Q'=S$. The bounded right-hand sides in the other
integral identities give
$S,F,G\in W^{1,\infty}([0,L])$. 

On a constant-$m$ interval, the drift of every polynomial moment is a
finite linear combination of polynomial moments, possibly involving one
higher spatial derivative. All such moments are initially continuous by
the spatial-derivative bounds and continuity of $Y$. Hence their integral
identities make all polynomial moments $C^1$ simultaneously. Repeating
the same argument, if all polynomial moments are $C^r$, then the drift of
each is $C^r$, so all are $C^{r+1}$. Thus every polynomial moment is
smooth on the interval. In particular, $Q$ is smooth there.
Differentiating $Q''=F-2mG$ and using
\eqref{eq:ito-F}--\eqref{eq:ito-G} gives \eqref{eq:ito-J}.
\end{proof}

In the rescaled variable, the Parisi functional is
\begin{equation}\label{eq:parisi-functional-time}
 \mathcal P_{\beta,h}(\nu)=\log2+U(0,h)
 -\frac1{2L}\int_0^L tm(t)\dd t.
\end{equation}
We next recall the variational characterization of
\cite[Theorem~2 and Proposition~1]{Chen2017Variational}, obtained
independently in
\cite[Proposition~1.1 and Section~3]{JagannathTobasco2017}.
The fixed-point identity at support points goes back to
\cite[Proposition~3.2]{Talagrand2006ParisiMeasures}; see also
\cite[Theorem~5]{AuffingerChen2015Properties}. We give the specialization
needed here, including the endpoint consequences.

\begin{lemma}\label{lem:variational-support}
Let $h\ge0$. If $\nu$ minimizes \eqref{eq:parisi-functional-time}, then its support
is contained in the set of global minima of
\begin{equation}\label{eq:variational-potential}
 \mathscr V(q)=\frac12\int_{Lq}^L\bigl(Q(t)-t/L\bigr)\dd t,
 \qquad 0\le q\le1.
\end{equation}
Then $1\notin\supp\nu$, and if $h>0$, also $0\notin\supp\nu$.
At each $q\in\supp\nu\cap(0,1)$,
\begin{equation}\label{eq:support-self-consistency}
 Q(Lq)=q,\qquad S(Lq)\le L^{-1}.
\end{equation}
\end{lemma}

\begin{proof}
Let $\widetilde\nu$ be a probability measure, and put
\[
 \widetilde m(t)=\widetilde\nu([0,t/L]),\qquad
 \delta m(t)=\widetilde m(t)-m(t),\qquad
 m_\varepsilon(t)=m(t)+\varepsilon\delta m(t).
\]
Let $U_\varepsilon$ correspond to
$\nu_\varepsilon=(1-\varepsilon)\nu+\varepsilon\widetilde\nu$,
and write $B_\varepsilon=\partial_xU_\varepsilon$.
Subtracting the two PDEs shows that
$V_\varepsilon=(U_\varepsilon-U)/\varepsilon$ satisfies
\[
 (V_\varepsilon)_t+\frac12(V_\varepsilon)_{xx}
 +\frac12m_\varepsilon(B_\varepsilon+B)(V_\varepsilon)_x
 =-\frac12\delta mB^2,
 \qquad V_\varepsilon(L,\cdot)=0.
\]
For $0\le s\le L$ and $x\in\R$, set
\[
 b_\varepsilon(r,x)=\frac12m_\varepsilon(r)
        [B_\varepsilon(r,x)+B(r,x)]
\]
and let
\[
 \dd Z_r^\varepsilon=b_\varepsilon(r,Z_r^\varepsilon)\dd r+\dd W_r,
 \qquad Z_s^\varepsilon=x.
\]
The coefficient is bounded and uniformly Lipschitz in space. The stochastic
representation of the linear equation is
\begin{equation}\label{eq:variation-stochastic-representation}
 V_\varepsilon(s,x)=\frac12\E
    \int_s^L\delta m(r)B(r,Z_r^\varepsilon)^2\dd r.
\end{equation}
To justify this representation, approximate $m$ and $\widetilde m$ in
$L^1$ by nondecreasing step functions on common finite partitions, form
the corresponding quantities in the display above, and apply It\^o's
formula on each interval between jump times, as in
\zcref{lem:ito-moments}. 
For each fixed $\varepsilon>0$, the stability estimate
\eqref{eq:profile-stability} gives uniform convergence of
$B_{\varepsilon,n}$ and $B_n$ to $B_\varepsilon$ and $B$, respectively.
Hence $b_{\varepsilon,n}\to b_\varepsilon$ in $L^1$ in time, uniformly
in space. Coupling the corresponding diffusions with the same Brownian
motion and applying Gronwall's inequality gives uniform convergence of
their paths. 
The source terms
$\delta m_n(r)B_n(r,\cdot)^2$ are uniformly bounded and converge in
$L^1$ in time, uniformly in space. Passing to the limit therefore gives
\eqref{eq:variation-stochastic-representation} for the original
coefficients.

From now on take $s=0$ and $x=h$ in
\eqref{eq:variation-stochastic-representation}, and write
$Z^\varepsilon$ for the resulting diffusion.
As $\varepsilon\downarrow0$, \eqref{eq:profile-stability} implies that
$b_\varepsilon\to mB$ in $L^1$ in time, uniformly in space. Coupling
$Z^\varepsilon$ with $Y$ using the same Brownian motion therefore gives
\[
 \sup_{0\le r\le L}|Z_r^\varepsilon-Y_r|\longrightarrow0.
\]
Since $B$ is uniformly Lipschitz in space,
\[
 B(r,Z_r^\varepsilon)^2\longrightarrow B(r,Y_r)^2
\]
uniformly in $r$. Using $|\delta m|\le1$ and dominated convergence in
\eqref{eq:variation-stochastic-representation} yields
\[
 V_\varepsilon(0,h)\longrightarrow
 \frac12\int_0^L\delta m(r)Q(r)\dd r.
\] 
Together with the explicit linear term in the Parisi functional, this yields
\[
 \left.\frac{d}{d\varepsilon}\mathcal P_{\beta,h}(\nu_\varepsilon)
 \right|_{\varepsilon=0+}
 =\frac12\int_0^L\bigl(Q(t)-t/L\bigr)\delta m(t)\dd t
 =\int\mathscr V\dd(\widetilde\nu-\nu).
\]
At a minimizer this is nonnegative for every $\widetilde\nu$.
Taking $\widetilde\nu=\delta_q$ shows that
\[
 \mathscr V(q)\ge c,\qquad
 c=\int\mathscr V\dd\nu,
\]
for every $q\in[0,1]$. Since $\mathscr V-c\ge0$ everywhere and
\[
 \int(\mathscr V-c)\dd\nu=0,
\]
we have $\mathscr V=c$ $\nu$-almost everywhere. Continuity of
$\mathscr V$ then gives $\mathscr V=c$ on $\supp\nu$.

Since $Q'=S$ and $S$ is continuous,
\[
 \mathscr V'(q)=\frac L2\bigl(q-Q(Lq)\bigr),\qquad
 \mathscr V''(q)=\frac L2\bigl(1-LS(Lq)\bigr).
\]
Since $Q(L)=\E\tanh^2Y_L<1$, one has $\mathscr V'(1)>0$, so $1$
is not a global minimum. If $h>0$, strict convexity and evenness of
$U(0,\cdot)$ imply $B(0,h)>0$ and hence $\mathscr V'(0)<0$; this
also excludes zero. Every positive support point is an interior minimum,
which gives \eqref{eq:support-self-consistency}.
\end{proof}

\zcref{thm:derivative-moments} is a positive-field counterpart of the
zero-field inequality in \cite[Proposition~9.1]{Lopatto2026}.
Here the additional hypothesis $N(t)\le0$ provides the parameter bounds
needed for a noncentered diffusion law.

For a fixed $m_0\in[0,1]$, define
\[
 J(t;m_0)=\E\bigl[C_{xx}^2-12m_0CC_x^2+6m_0^2C^4\bigr],
\]
where the spatial derivatives are evaluated at $(t,Y_t)$. The theorem below
allows this value $m_0$ to differ from the value assigned to $m(t)$.
Either choice $m_0=m(t-)$ or $m_0=m(t+)$ satisfies the two
one-sided coefficient inequalities in the theorem.
The PDE solution and diffusion still correspond to the original coefficient $m$.

\begin{theorem}\label{thm:derivative-moments}
Let $m:[0,L]\to[0,1]$ be nondecreasing and measurable, let $0<t<L$
and $h>0$, and let $m_0\in[0,1]$ satisfy
\[
 m(s)\le m_0\quad\text{for a.e. }s\in[0,t),\qquad
 m(s)\ge m_0\quad\text{for a.e. }s\in(t,L].
\]
For the PDE solution and diffusion determined by $m$, define
\[
 N(t)=Q(t)-tS(t).
\]
The implication
\begin{equation}\label{eq:strict-derivative-moment-inequality}
 N(t)\le0,\qquad F(t)-2m_0G(t)=0
 \quad\Longrightarrow\quad J(t;m_0)>0
\end{equation}
holds.
\end{theorem}

Sections~\ref{sec:profiles}--\ref{sec:positivity} prove
\zcref{thm:derivative-moments}. The coefficient bound after $t$ gives the
normalized PDE inequalities in \zcref{sec:profiles}, while the bound
before $t$ enters the density comparison in \zcref{sec:diffusion-density}.
\zcref{sec:gaussian-comparison} uses $N(t)\le0$ to control the Gaussian
parameters. Sections~\ref{sec:weighted-moments}--\ref{sec:positivity}
combine these inputs to prove the strict moment inequality, whose proof
is completed in \zcref{sec:strict-inequality-completion}.
\zcref{sec:structure} then deduces \zcref{thm:main}.

\section{Estimates for the Parisi PDE}\label{sec:profiles}
The density comparison and the integrations by parts require spatial
derivative estimates that remain valid for arbitrary nondecreasing
coefficients in the original Parisi time coordinate, including coefficients
with jumps corresponding to atoms of the Parisi measure. We first prove
uniform regularity estimates and then establish the differential inequalities
preserved by the PDE after passing to reverse time.

We use reverse time throughout this section. Starting from the original
Parisi time coordinate in \eqref{eq:parisi-time}, set
\[
 \tau=L-t,\qquad
 U_{\mathrm{rev}}(\tau,x)=U(L-\tau,x),\qquad
 m_{\mathrm{rev}}(\tau)=m(L-\tau).
\]
The terminal-value Parisi PDE then becomes the initial-value equation
\begin{equation}\label{eq:reverse-profile-pde}
 \partial_\tau U_{\mathrm{rev}}
 =\frac12\bigl((U_{\mathrm{rev}})_{xx}
       +m_{\mathrm{rev}}(\tau)(U_{\mathrm{rev}})_x^2\bigr),
 \qquad
 U_{\mathrm{rev}}(0,x)=\log\cosh x.
\end{equation}
Since $m$ is nondecreasing in the original Parisi time coordinate,
$m_{\mathrm{rev}}$ is nonincreasing. In this section only, we suppress
the subscript ``$\mathrm{rev}$'' and write $U,m$ for these reverse-time
quantities, with $B=U_x$ and $C=U_{xx}$.

We collect standard regularity and approximation estimates for the Parisi
PDE; see \cite[Propositions~1--2]{AuffingerChen2015Properties} and
\cite[Theorem~4 and Lemmas~14, 16]{JagannathTobasco2016Dynamic}. The proof
below also records the weighted spatial-derivative bounds and the $L^1$
stability estimates for every fixed derivative order used later.

\begin{lemma}\label{lem:pde-regularity}
For every \(0<T<\infty\) and measurable \(m:[0,T]\to[0,1]\),
\eqref{eq:reverse-profile-pde} has a unique solution in the mild sense whose
derivative \(B\) is bounded. For every compact $K\subset\R$ and every integer
$j\ge0$, the maps
\[
 \tau\longmapsto U(\tau,\cdot)|_K,\qquad
 \tau\longmapsto \partial_x^jB(\tau,\cdot)|_K
\]
are absolutely continuous with values in $C(K)$. The solution is even in
\(x\), smooth in \(x\), and satisfies
\[
 B(\tau,-x)=-B(\tau,x),\qquad
 0<C(\tau,x)\le 1-B(\tau,x)^2,\qquad
 \lim_{x\to\pm\infty}B(\tau,x)=\pm1 .
\]
For every integer \(j\ge0\), the spatial derivatives
\(\partial_x^j B\) are bounded on \([0,T]\times\mathbb R\).  Their time
derivatives exist almost everywhere, are bounded, and satisfy the
differentiated equations.  There exist constants \(0<c_T<C_T<\infty\),
independent of $m$, such that
\begin{equation}\label{eq:profile-two-sided-tails}
 c_T\operatorname{sech}^2x
 \le C(\tau,x)\le C_T\operatorname{sech}^2x
 \qquad(0\le\tau\le T).
\end{equation}
For every integer $k\ge0$, there exists a constant $K_{T,k}$, independent
of $m$, such that
\begin{equation}\label{eq:profile-weighted-derivatives}
 |\partial_x^k C(\tau,x)|\le K_{T,k}\sech^2x,
 \qquad \left|\frac{\partial_x^kC(\tau,x)}{C(\tau,x)}\right|\le K_{T,k}.
\end{equation}
Moreover $U\ge0$ and $C_x\le0$ on the positive half-line.
For two coefficients $m_n,m$, let $U_n,U$ be the corresponding solutions
and put $B_n=(U_n)_x$, $B=U_x$. Then every integer $j\ge0$ satisfies
\begin{equation}\label{eq:profile-stability}
 \sup_{\tau\le T}\|\partial_x^j(B_n-B)(\tau)\|_\infty
 \le K_{T,j}\|m_n-m\|_{L^1[0,T]}.
\end{equation}
Moreover, $U_n\to U$ locally uniformly together with their spatial
derivatives whenever $m_n\to m$ in $L^1[0,T]$.
\end{lemma}

\begin{proof}
The bound $C+B^2\le1$ gives global existence. We then use the weight
$\sech^2x$ to bound $C$, its spatial derivatives, and their ratios to $C$.
Let $\mathsf P_t$ be the heat semigroup. The Burgers equation
\begin{equation}\label{eq:profile-burgers}
 B_\tau=\frac12B_{xx}+mBB_x,\qquad B(0,x)=\tanh x,
\end{equation}
has the mild form
\[
 B(\tau)=\mathsf P_\tau\tanh+
 \frac12\int_0^\tau m(s)\partial_x\mathsf P_{\tau-s}(B(s)^2)\dd s.
\]
The estimate $\|\partial_x\mathsf P_t f\|_\infty
\le\sqrt{2/(\pi t)}\|f\|_\infty$ makes the integral map a contraction
on a bounded ball in $C([0,T_0];\mathrm{BUC}(\R))$ for small $T_0$. Using the
nondivergence form and differentiating the heat semigroup gives the same
construction in $C_b^1$, since the singular kernel $t^{-1/2}$ is integrable.
Work first with smooth coefficients $m$. The maximum principle gives
$|B|\le1$, and differentiation yields
\begin{align}
 C_\tau&=\frac12C_{xx}+mBC_x+mC^2,\label{eq:curvature-pde}\\
 (C+B^2)_\tau
 &=\frac12(C+B^2)_{xx}+mB(C+B^2)_x+(m-1)C^2.
 \label{eq:curvature-plus-square}
\end{align}
The first equation, regarded locally as linear with potential $mC$,
preserves $C>0$; the second gives $C+B^2\le1$. These bounds continue
the local solution to every finite time. Reflection gives parity.
All comparisons below may first be made on bounded spatial intervals:
add $\varepsilon e^{a\tau}(1+x^2)$ to a supersolution, choose $a$ larger
than the bounds on the drift and reaction coefficients, and then let the
interval grow and $\varepsilon$ decrease to zero. This gives comparison
for bounded solutions on the line and on the half-line with the stated
boundary values.

Put $\phi(x)=\sech^2x$. Since
$\phi'/\phi=-2\tanh x$ and $\phi''/(2\phi)=2-3\phi$,
multiples of $e^{A\tau}\phi$ are comparison barriers for
$\partial_\tau-(1/2)\partial_{xx}-b\partial_x-a$ whenever $a,b$
are bounded. More precisely, a solution with initial data bounded by
$K\phi$ and forcing bounded by $K\phi$ has absolute value bounded by
$K_T\phi$; choose $A>2+2\|b\|_\infty+\|a\|_\infty$ and increase
the prefactor. Applied to the equation for $C$, the same comparison gives
\[
 e^{-5\tau}\phi(x)\le C(\tau,x)\le e^{5\tau}\phi(x).
\]
For ${C_k}=\partial_x^kC$, $k\ge1$, differentiation gives
\[
 ({C_k})_\tau=\frac12({C_k})_{xx}+mB({C_k})_x+m(k+2)C{C_k}+\mathscr R_k,
\]
where
\[
 \mathscr R_k=
 m\sum_{l=2}^k\binom kl C_{l-1}C_{k-l+1}
 +m\sum_{l=1}^{k-1}\binom kl C_lC_{k-l}.
\]
Only lower derivatives occur in $\mathscr R_k$. The initial derivatives
are bounded by constants times $\phi$, and $\phi^2\le\phi$.
Induction with the same barriers gives
$|{C_k}|\le K_{T,k}\phi$, and the lower bound for $C$ gives
\eqref{eq:profile-weighted-derivatives}. This also bounds all spatial derivatives
of $B$; their equations bound the time derivatives.
In particular $B_\tau=C_x/2+mBC=O_T(\phi)$, so integration in time
gives $B(\tau,x)-\tanh x=O_T(\phi(x))$ and the limits at infinity.
For $k=1$ the source $\mathscr R_1$ vanishes. Its nonpositive initial
data on $x\ge0$ and zero boundary value give $C_x\le0$.
Also $U_\tau=(C+mB^2)/2\ge0$, hence $U\ge\log\cosh x\ge0$.

For stability put $w=B_n-B$. Subtraction gives
\[
 w_\tau=\frac12w_{xx}+m_nB_nw_x+m_nCw+(m_n-m)BC,\qquad w(0)=0.
\]
Since $0\le m_nC\le1$ and $|BC|\le1$, comparison gives
$\sup_{\tau\le T}\|w(\tau)\|_\infty\le e^T\|m_n-m\|_{L^1}$.
After $j$ spatial differentiations the drift is $m_nB_n$, the
zero-order coefficient is $m_n(jC_n+C)$, and the remaining terms
involve lower derivatives of $w$ or
$(m_n-m)\partial_x^j(BC)$. Induction proves \eqref{eq:profile-stability}.
Finally
\[
 U(\tau,0)=\frac12\int_0^\tau C(s,0)\dd s,\qquad
 U(\tau,x)=U(\tau,0)+\int_0^xB(\tau,y)\dd y
\]
give convergence of $U$. For measurable $m$, choose smooth
$m_n\to m$ in $L^1[0,T]$. Applying \eqref{eq:profile-stability} pairwise gives
\[
 \sup_{\tau\le T}\|\partial_x^j(B_n-B_k)(\tau)\|_\infty
 \le K_{T,j}\|m_n-m_k\|_{L^1[0,T]}.
\]
Thus every spatial derivative is Cauchy in $C([0,T];C_b(\R))$; successively
passing to the limit identifies these limits with the spatial derivatives of
$B$. The differentiated right-hand sides converge in $L^1([0,T];C_b(\R))$
and are uniformly bounded, so their time-integral forms pass to the limit,
giving the asserted time regularity and differentiated equations for
measurable $m$.
\end{proof}

We next compare the spatial derivative of $-C_x/(2C)$ with $m_0C$,
where $m_0$ is a lower bound for the coefficient in reverse time.
The inequalities below control their nonnegative difference and show
that its ratio to $C$ is nondecreasing. These comparisons will give
the normalized inequalities used in \zcref{sec:weighted-moments}.
They are the inequalities of
\cite[Propositions~3.6--3.7]{Lopatto2026} in the original spatial
coordinate; we give a comparison proof that also includes $m_0=0$.

\begin{lemma}\label{lem:differential-inequalities}
Suppose $m$ is nonincreasing in the reverse time coordinate. At a fixed
time $\tau$, choose $m_0\in[0,1]$ no larger than $m$ almost everywhere
on $[0,\tau]$, and put
\[
 \zeta=-\frac{C_x}{2C},\qquad \mathsf f=\zeta_x-m_0C.
\]
Then, for $x\ge0$,
\begin{equation}\label{eq:differential-inequalities}
 \zeta\ge0,\qquad \mathsf f\ge0,\qquad
 -2\zeta\mathsf f\le\mathsf f_x\le\zeta\mathsf f,
 \qquad \left(\frac{\mathsf f}{C}\right)_x\ge0.
\end{equation}
Returning to the original time coordinate $t=L-\tau$, the hypothesis becomes
\[
 m(s)\ge m_0\qquad\text{for almost every }s\in[t,L].
\]
In particular, at a support endpoint one may take $m_0$ to be the value
approached from within the support interval.
\end{lemma}

\begin{proof}
By \zcref{lem:pde-regularity}, $C_x\le0$ on $x\ge0$, so
$\zeta\ge0$ there. First take a smooth nonincreasing coefficient $m$
in reverse time and use its value $m=m(\tau)$. Define
\[
 \mathsf f=\zeta_x-mC,\qquad
 \mathsf g=\mathsf f_x+2\zeta\mathsf f,\qquad
 \mathsf p=\zeta\mathsf f-\mathsf f_x,
\]
and write
$\mathscr L_{\mathrm{rev}}=\partial_\tau-(1/2)\partial_{xx}-(mB-2\zeta)\partial_x$.
Differentiation of \eqref{eq:curvature-pde} first gives
\begin{equation}\label{eq:log-derivative-evolution}
 \zeta_\tau=\frac12\zeta_{xx}+(mB-2\zeta)\zeta_x+2mC\zeta.
\end{equation}
Using $C_x=-2\zeta C$ and $\zeta_x=\mathsf f+mC$, further
differentiation gives the triangular system
\begin{align}
 \mathscr L_{\mathrm{rev}}\mathsf f
   &=-(2\mathsf f+mC)\mathsf f-\dot mC,
       \label{eq:physical-f-evolution}\\
 \mathscr L_{\mathrm{rev}}\mathsf g
   &=-(8\mathsf f+4mC)\mathsf g
       +12\zeta\mathsf f(\mathsf f+mC),
       \label{eq:physical-g-evolution}\\
 \mathscr L_{\mathrm{rev}}\mathsf p
   &=-(5\mathsf f+mC)\mathsf p
       +3\zeta\mathsf f^2-3\dot m\zeta C.
       \label{eq:physical-p-evolution}
\end{align}
Each fixed spatial derivative of $\zeta=-C_x/(2C)$ is a polynomial in
ratios of derivatives of $C$ to $C$. Thus
\eqref{eq:profile-weighted-derivatives} bounds all coefficients and unknowns
in this system on a finite time interval. For a smooth approximation,
$\dot m$ is bounded as well. Initially $C=\sech^2x$, $\zeta=\tanh x$, and
\[
 \mathsf f=(1-m(0))C\ge0,\qquad \mathsf g=0,\qquad
 \mathsf p=3(1-m(0))\tanh x\sech^2x\ge0\quad(x\ge0).
\]
Since $-\dot m\ge0$, the first equation has nonnegative source
$-\dot m C$ and a bounded reaction coefficient $-(2\mathsf f+mC)$.
Comparison gives $\mathsf f\ge0$ on $\R$. Its sign makes the source
in the second equation nonnegative on $x\ge0$, so $\mathsf g\ge0$.
The source in the third is then nonnegative as well, giving
$\mathsf p\ge0$. Both odd quantities vanish at zero.
Each equation has diffusion coefficient $1/2$; comparison on the
unbounded domain follows by a quadratic spatial barrier and an
exponential time weight.
The three signs give the first three inequalities in
\eqref{eq:differential-inequalities}, and
$(\mathsf f/C)_x=\mathsf g/C\ge0$ gives the last.

At the fixed time $\tau$, keep $B$ and $C$ unchanged and replace $m(\tau)$
by $m_0\le m(\tau)$ in the definitions of these quantities. With
$\delta=m(\tau)-m_0\ge0$, this changes them by
\[
 \mathsf f\mapsto\mathsf f+\delta C,\qquad
 \mathsf g\mapsto\mathsf g,\qquad
 \mathsf p\mapsto\mathsf p+3\delta\zeta C.
\]
Thus their signs persist. Approximate an arbitrary nonincreasing $m$
by smooth nonincreasing functions preserving the lower bound $m_0$.
The spatial-derivative convergence in \zcref{lem:pde-regularity} and
positivity of $C$ on compact intervals pass the inequalities to the
limit, including when $m_0=0$.
\end{proof}

For the Gaussian comparison we need $\zeta(x)/x$ to be nonincreasing.
The next lemma also proves positivity of the normalization parameter
$\lambda$.

\begin{lemma}\label{lem:log-derivative}
Suppose $m$ is nonincreasing in reverse time. At a fixed reverse time
$\tau$, write $C(x)=C(\tau,x)$ and put $\zeta=-C_x/(2C)$. Then
\begin{equation}\label{eq:log-derivative-radial}
 x\zeta_x(x)-\zeta(x)\le0\qquad(x>0).
\end{equation}
Moreover, with $c_0=C(0)$ and
$\lambda^2=-C_{xx}(0)/(2c_0)$, one has $\lambda>0$ and
\begin{equation}\label{eq:log-derivative-bound}
 0<\zeta_x(x)\le\lambda^2\qquad(x\in\R).
\end{equation}
\end{lemma}

\begin{proof}
\zcref{lem:differential-inequalities} with $m_0=0$ gives
$\zeta_x\ge0$. Since $C_x\le0$, one has
$B(x)=\int_0^xC(y)\dd y\ge xC(x)$ for $x\ge0$.
Set $H=x\zeta_x-\zeta$. Its sign determines the monotonicity of
$\zeta(x)/x$. The weighted derivative bounds allow $H$ to grow with $x$,
so we use a quadratic barrier. For smooth nonincreasing $m$, differentiating 
\eqref{eq:log-derivative-evolution} gives the equation for $H$:
\begin{align}\label{eq:radial-quantity-evolution}
 H_\tau={}&\frac12H_{xx}
 +(mB-2\zeta-x^{-1})H_x+2(mC-\zeta_x)H+m(xC-B)\zeta_x-4mxC\zeta^2.
\end{align}
Both source terms are nonpositive. Initially
$H=x\sech^2x-\tanh x\le0$, and $H(\tau,0)=0$.
The bounds on ratios of spatial derivatives to $C$ give
$|H(\tau,x)|\le K_T(1+x)$. Since $\zeta$ is odd and smooth, its Taylor
expansion at zero has the form
\[
 \zeta(\tau,x)=a(\tau)x+b(\tau)x^3+O_T(x^5),
 \qquad H(\tau,x)=2b(\tau)x^3+O_T(x^5).
\]
Fix $\varepsilon>0$ and set
\[
 \Phi_\varepsilon(s,x)=\varepsilon e^{As}(1+x^2).
\]
Choose $A$ sufficiently large that $\Phi_\varepsilon$ is a strict
supersolution on $x>0$. Although the drift contains $-x^{-1}$, the
corresponding term
\[
 -x^{-1}\partial_x\Phi_\varepsilon=-2\varepsilon e^{As}
\]
remains bounded at zero. Apply comparison first on
$[0,\tau]\times[\delta,R]$. For $\delta$ sufficiently small, the Taylor
estimate gives $H(s,\delta)\le\Phi_\varepsilon(s,\delta)$ uniformly in
$s$, while for $R$ sufficiently large the linear growth bound on $H$
gives $H(s,R)\le\Phi_\varepsilon(s,R)$. The initial inequality follows
from $H(0,x)\le0$. Hence
\[
 H(s,x)\le\Phi_\varepsilon(s,x).
\]
Letting $\delta\downarrow0$, $R\uparrow\infty$, and then
$\varepsilon\downarrow0$ proves
\eqref{eq:log-derivative-radial}. Smooth monotone approximation
extends it to every nonincreasing coefficient $m$.

Thus $\zeta(x)/x$ is nonincreasing on $(0,\infty)$, and
\[
 0\le\zeta_x(x)\le\frac{\zeta(x)}x
 \le\zeta_x(0)=\lambda^2.
\]
If $\lambda=0$, then $\zeta$ vanishes identically, contradicting
$C>0$ and $C(x)\to0$. Hence $\lambda>0$.
The last inequality of \eqref{eq:differential-inequalities}, with $m_0=0$,
gives $\zeta_x/C\ge\lambda^2/c_0$ on $x\ge0$. This proves strict
positivity of $\zeta_x$; parity handles $x<0$.
\end{proof}

We now express the preceding differential inequalities in normalized
variables and record their elementary consequences.

\begin{lemma}\label{lem:normalized-estimates}
Let \(U\) solve \eqref{eq:reverse-profile-pde} with a nonincreasing
coefficient $m$, and evaluate its spatial derivatives at a fixed reverse
time $\tau$. Write $B(x)=U_x(\tau,x)$ and $C(x)=U_{xx}(\tau,x)$, and set
\[
 c_0=C(0),\qquad
 \lambda^2=-\frac{C_{xx}(0)}{2c_0},\qquad
 v=\lambda x,\qquad A(v)=C(v/\lambda)/c_0,\qquad
 u(v)=-\frac{C_x(v/\lambda)}{2\lambda C(v/\lambda)} .
\]
Then \(\lambda>0\), \(A'=-2Au\), \(u(0)=0\), \(u'(0)=1\), and, for \(v\ge0\),
\begin{equation}\label{eq:normalized-profile-estimates}
\begin{gathered}
 0<A\le u'\le1,\qquad
 \tanh v\le u(v)\le v,\qquad vu'(v)\le u(v),\\
 1-A\le u^2,\qquad
 \lim_{v\to\infty}u(v)=1/\lambda .
\end{gathered}
\end{equation}

Now fix \(m_0\in(0,1]\) no larger than $m$ almost everywhere on $[0,\tau]$.
Define functions of the normalized coordinate $v$ by
\[
 a(v)=\frac{m_0C(v/\lambda)}{\lambda^2},\qquad
 j(v)=u'(v)-a(v),\qquad b(v)=\frac{m_0B(v/\lambda)}{\lambda},\qquad
 \mu=\frac{m_0c_0}{\lambda^2}.
\]
Then \(0<\mu\le1\), \(a(0)=\mu\), \(j(0)=1-\mu\), and, for \(v\ge0\),
\begin{align}
 &a'=-2au,\qquad u'=a+j,\qquad b'=a,\qquad
 -2uj\le j'\le uj,\label{eq:normalized-differential-inequalities}\\
 &j\ge ka,\qquad k=\frac{1-\mu}{\mu},\qquad
 0<u'\le1,\qquad \tanh v\le u(v)\le v,
 \label{eq:normalized-log-derivative-bounds}\\
 &a\ge\mu(1-u^2),\qquad
 0\le b\le\mu u,\qquad
 b^2\le\mu^2(1-a/\mu).
 \label{eq:normalized-magnetization-bounds}
\end{align}
Returning to the original time coordinate $t=L-\tau$, the estimates involving
$a,j,b,\mu,k$ apply whenever
\[
 m(s)\ge m_0\qquad\text{for almost every }s\in[t,L].
\]
In particular, they apply with $m_0$ equal to either positive one-sided
support value approached from the interior.
\end{lemma}

\begin{proof}
\zcref{lem:log-derivative} gives $\lambda>0$ and $u'\le1$.
Apply \zcref{lem:differential-inequalities} with $m_0=0$, so
$\mathsf f=\zeta_x$ and $\mathsf f/C\ge\lambda^2/c_0$ on $x\ge0$.
With unnormalized functions evaluated at $x=v/\lambda$, this gives
\[
 u'=\frac{\zeta_x}{\lambda^2}\ge A>0.
\]
Integrating $u'\le1$ gives $u\le v$ on $v\ge0$, and
\eqref{eq:log-derivative-radial} gives $vu'\le u$.
The function $e=u^2-1+A$ vanishes at zero and satisfies
\[
 e'=2u(u'-A)\ge0.
\]
Thus $u'\ge A\ge1-u^2$, and scalar ODE comparison gives
$u\ge\tanh v$. Finally \eqref{eq:profile-two-sided-tails} implies
\[
 -\log A(v)=2v/\lambda+O(1)=2\int_0^vu(s)\dd s.
\]
Since $u$ is increasing, this Ces\`aro limit gives
$\lim_{v\to\infty}u(v)=1/\lambda$.

For a fixed $m_0>0$ as above, put $\mathsf f=\zeta_x-m_0C$.
Then
\[
 \lambda^2=m_0c_0+\mathsf f(0),\qquad
 a=\frac{m_0C}{\lambda^2},\qquad
 j=\frac{\mathsf f}{\lambda^2},\qquad u'=a+j.
\]
Hence $0<\mu\le1$ and $j(0)=1-\mu$. The differential inequalities
in \eqref{eq:differential-inequalities} give $-2uj\le j'\le uj$, while
\[
 \frac ja=\frac{\mathsf f}{m_0C}
 \ge\frac{\mathsf f(0)}{m_0c_0}=\frac{1-\mu}{\mu}.
\]
The remaining bounds on $u$ were already proved independently of
$m_0$, and $a=\mu A\ge\mu(1-u^2)$.
Finally $a\le\mu u'$ and $b'=a$, with $b(0)=0$, give
$0\le b\le\mu u$. Therefore
\[
 (b^2)'=2ba\le2\mu ua=-\mu a',
 \qquad b^2\le\mu(\mu-a)=\mu^2(1-a/\mu).
\]
\end{proof}

\section{The diffusion density}\label{sec:diffusion-density}

We estimate the diffusion density at a fixed time before weighting it by
$C^2$. Dividing the transformed density by the Neumann heat kernel
extracts the explicit heat-kernel contribution of the starting point $h$.
The remaining ratio may still depend on $h$; we prove that its negative
logarithm is nondecreasing and concave as a function of squared position.

For the rest of the manuscript, we use the original Parisi time $t$
and the notation $U,m,B,C$ of \zcref{sec:variation}. Fix $0<t\le L$,
$h\ge0$, and $m_*\in[0,1]$ such that the nondecreasing coefficient on
$[0,t)$ is at most $m_*$ and the coefficient on $(t,L]$ is at least
$m_*$, almost everywhere. Keep the restriction of $m$ to $(t,L]$,
and hence the resulting solution $U(t,\cdot)$, fixed. On $[0,t)$ use
the right-continuous representative of $m$ and assign $m(t)=m_*$. Its
Stieltjes measure $dm$ on $(0,t]$ includes any final jump. Let $p_h(t,x)$
be the density of $Y_t$ for
\begin{equation}\label{eq:mc-lower-sde}
 \dd Y_r=m(r)B(r,Y_r)\dd r+\dd W_r,\qquad Y_0=h.
\end{equation}

We use the Brownian-bridge representation and the construction of the
transformed density in \cite[Proposition~7.1, (7.2)]{Lopatto2026}, with
initial point $h$ and rescaled time. The additional estimate is concavity of the
negative logarithm of the ratio to the Neumann heat kernel, as a function
of squared position. We prove the estimate first for smooth nondecreasing
coefficients and then pass to the limit in the bridge representation.

\begin{lemma}\label{lem:lower-density}
In the preceding setting, $p_h(t,\cdot)$ is strictly positive and smooth.
Let $K_t(x)=(2\pi t)^{-1/2}e^{-x^2/(2t)}$, and define $k_h$ and $\Psi$ by
\begin{equation}\label{eq:transformed-density-ratio}
 \begin{aligned}
 k_h(t,x)&=K_t(x-h)+K_t(x+h),\\
 e^{-m_*U(t,x)}[p_h(t,x)+p_h(t,-x)]&=k_h(t,x)e^{-\Psi(t,x^2)}.
 \end{aligned}
\end{equation}
The function $k_h$ is the Neumann heat kernel on the half-line, extended
evenly in $x$. The function $\Psi(t,\cdot)$ is smooth on $[0,\infty)$ and
\begin{equation}\label{eq:lower-density-concavity}
 \Psi_r(t,r)\ge0,\qquad \Psi_{rr}(t,r)\le0\qquad(r\ge0),
\end{equation}
with one-sided derivatives at zero. More generally, for every $t_0>0$ and
positive integer $k$, the $k$th spatial derivative of $\Psi(t,x^2)$ is
bounded uniformly over $t_0\le t\le L$, $x\in\R$, and all coefficients $m$
satisfying the preceding monotonicity and endpoint conditions. The bound
depends only on $L,h,k,t_0$.

For any $\lambda>0$, weighting the law of $|Y_t|$ by $C(t,Y_t)^2$ and
rescaling by $v=\lambda x$ gives
\begin{equation}\label{eq:mc-rho-density}
 \rho(\dd v)=\frac1{\mathcal Z}\omega(v)e^{-dv^2/2-\psi(v^2)}
                  \cosh(\alpha v)\dd v,\qquad v\ge0,
\end{equation}
where $0<\mathcal Z<\infty$ and
\begin{equation}\label{eq:mc-alpha-d}
 \begin{gathered}
 \omega(v)=C(t,v/\lambda)^2e^{m_*U(t,v/\lambda)},\qquad
 d=\frac1{\lambda^2t},\qquad \alpha=\frac h{\lambda t},\\
 \psi(r)=\Psi(t,r/\lambda^2).
 \end{gathered}
\end{equation}
Define the even function
\begin{equation}\label{eq:coefficient-s}
 s(v)=d+2\psi'(v^2).
\end{equation}
It satisfies
\begin{equation}\label{eq:coefficient-s-bounds}
 d\le s(v)\le s(0)<\infty,\qquad s'(v)\le0\qquad(v\ge0).
\end{equation}
\end{lemma}

\begin{proof}
The bridge formula supplies positivity, spatial regularity, and convergence
under approximation. We establish it first and then prove the two signs in
\eqref{eq:lower-density-concavity} by a maximum principle.

\medskip
\emph{Step 1: The bridge representation.}
First assume that $m$ is smooth on the time interval under consideration,
while the coefficient on $(t,L]$ and hence the terminal function
$U(t,\cdot)$ are kept fixed. The differentiated PDE and
\zcref{lem:pde-regularity} give the time and spatial regularity required for
classical It\^o's formula. Along reference Brownian motion
$\mathsf X_r=h+W_r$ it gives
\[
 \dd U(r,\mathsf X_r)=B(r,\mathsf X_r)\dd W_r
                 -\frac12m(r)B(r,\mathsf X_r)^2\dd r.
\]
Here $B$ is bounded and $U$ has at most linear spatial growth, so localization
and passage to the interval endpoints are justified. The product rule for
$mU$ therefore gives
\[
 \int_0^t mB\dd W-\frac12\int_0^t m^2B^2\dd r
 =m_*U(t,\mathsf X_t)-m(0)U(0,h)
       -\int_{(0,t]}U(r,\mathsf X_r)\dd m(r).
\]
The left side is the logarithm of the Girsanov density. Since $|mB|\le1$,
Novikov's condition holds. Let $\E_{h\to x}^{\mathrm{bridge}}$ denote
expectation for a Brownian bridge on $[0,t]$ from $h$ to $x$.
Conditioning on $\mathsf X_t=x$ yields
\begin{equation}\label{eq:lower-bridge-formula}
 \begin{split}
 p_h(t,x)={}&K_t(x-h)e^{m_*U(t,x)-m(0)U(0,h)} \E_{h\to x}^{\mathrm{bridge}}
       \exp\left(-\int_{(0,t]}U(r,\mathsf X_r)\dd m(r)\right).
 \end{split}
\end{equation}
The same formula holds for a nondecreasing coefficient of bounded variation.
To justify the passage, extend $\dd m$ from $(0,t]$ to a finite measure
on $[0,t]$ by assigning zero mass to $\{0\}$, and choose smooth
nondecreasing approximations such that 
$\dd m_n\Rightarrow\dd m$ weakly as finite measures on $[0,t]$, with the
prescribed initial value and total mass. Equivalently, the smooth densities
may be supported in $(0,t)$ while approximating a possible terminal atom.
Keep the coefficient on $(t,L]$ fixed in every approximation, so the terminal
function $U(t,\cdot)$ is unchanged. Then $m_n\to m$ in $L^1$, and the PDE
stability estimates and a coupling with the same Brownian motion give
convergence of the diffusion laws. On each continuous bridge path,
\[
 \begin{aligned}
 &\left|\int_{(0,t]}U_n(r,\mathsf X_r)\dd m_n(r)
       -\int_{(0,t]}U(r,\mathsf X_r)\dd m(r)\right|\\
 &\qquad\le
 \sup_{0\le r\le t}|U_n(r,\mathsf X_r)-U(r,\mathsf X_r)|
       \,[m_n(t)-m_n(0)]
 +\left|\int_{(0,t]}U(r,\mathsf X_r)\dd(m_n-m)(r)\right|.
 \end{aligned}
\]
The first term tends to zero because the path has compact range and the
solutions converge uniformly there; the second tends to zero by weak
convergence against the continuous function
$r\mapsto U(r,\mathsf X_r)$. The integrands in the bridge expectation are
at most one because $U_n\ge0$. The prefactor in
\eqref{eq:lower-bridge-formula} is bounded by a constant times
$K_t(x-h)e^{|x|}$, uniformly in the approximations. Dominated convergence
in $x$ identifies the limit with the density of $Y_t$. A final jump is
included in this argument: its contribution $\Delta mU(t,x)$ to the
Stieltjes integral cancels the corresponding change in
$e^{m_*U(t,x)}$.

A bridge from $h$ to $x$ has the form
$\mathsf X_r=(1-r/t)h+(r/t)x+\mathfrak b_r$, where
$\mathfrak b$ is a centered Brownian bridge independent of $x$. For
$V_x=\int_{(0,t]}U(r,\mathsf X_r)\dd m(r)$,
\[
 \partial_x^kV_x=\int_{(0,t]}(r/t)^k
       \partial_x^kU(r,\mathsf X_r)\dd m(r),\qquad
 |\partial_x^kV_x|\le C_{L,k}\quad(k\ge1).
\]
Here $m(t)-m(0)\le1$, and \zcref{lem:pde-regularity} bounds every fixed
spatial derivative. Derivatives of $e^{-V_x}$ are $e^{-V_x}$ times
polynomials in these bounded derivatives. Differentiating under the expectation therefore bounds each derivative of
the bridge expectation by a constant times the expectation itself. Its
logarithmic
derivatives are consequently bounded.

Oddness of the drift identifies $p_h(t,-x)$ with the density at $x$
of the same diffusion started at $-h$. Thus, for the ratio in
\eqref{eq:transformed-density-ratio}, combine the bridge factors
starting at $h$ and $-h$ with weights $e^{hx/t}$ and $e^{-hx/t}$,
and divide by $2\cosh(hx/t)$. Derivative ratios of a positive sum are
weighted averages of the corresponding derivative ratios. Thus all fixed
positive-order derivatives of $\Psi(t,x^2)$ are bounded, uniformly away
from $t=0$. The function is even and smooth in $x$. Taylor's formula for
an even smooth function gives a smooth function of $r=x^2$ on the closed
half-line, with one-sided derivatives at zero.

\medskip
\emph{Step 2: The differential inequalities.}
First assume that $m$ is smooth, nondecreasing, and constant near zero.
Use $\sigma$ for running forward time and put
\[
 g(\sigma,x)=e^{-m(\sigma)U(\sigma,x)}
                   [p_h(\sigma,x)+p_h(\sigma,-x)],\qquad
 \mathcal U(\sigma,r)=U(\sigma,\sqrt r).
\]
We first record the joint regularity needed below. Fix
$0<\delta<T<t$. For smooth $m$, represent the Brownian bridge from $h$ to
$x$ over time $\sigma$ on one unit-time bridge by
\[
 \mathsf X_{\sigma u}=(1-u)h+ux+\sqrt\sigma\,\mathfrak b_u,
 \qquad 0\le u\le1.
\]
After the change of variables $r=\sigma u$, the bridge representation,
the differentiated PDE, and the bounds in \zcref{lem:pde-regularity}
permit dominated differentiation in $(\sigma,x)$ locally on
$[\delta,T]\times\R$. Thus every spatial derivative of $p_h$ is jointly
continuous there. The SDE gives the forward equation in the weak sense; its
right-hand side and all required spatial derivatives are jointly continuous,
so the equation is classical and, for every $k\ge0$,
\begin{equation}\label{eq:density-joint-regularity}
 \partial_\sigma\partial_x^k p_h
 =\frac12\partial_x^{k+2}p_h
  -\partial_x^{k+1}\!\left(mB\,p_h\right)
 \qquad (\delta\le\sigma\le T).
\end{equation}
Consequently the substitutions and spatial differentiations below are
classical on every positive-time strip. Evenness and the spatial Taylor
expansions give the corresponding one-sided values after passing to
$r=x^2=0$.

Since $B$ is odd, the sum $p_h(\sigma,x)+p_h(\sigma,-x)$ satisfies
the same forward density equation as $p_h$. Substituting this sum as
$e^{mU}g$ gives
\[
 g_\sigma+(\dot m U+mU_\sigma)g
 =\frac12g_{xx}-\frac m2(C+mB^2)g.
\]
Using $U_\sigma=-(C+mB^2)/2$ cancels the terms containing $m$ and leaves
\begin{equation}\label{eq:transformed-density-evolution}
 g_\sigma=\frac12g_{xx}-\dot mUg.
\end{equation}
Since $k_h$ solves the heat equation, division by $k_h$ and passage to the
negative logarithm give, with $r=x^2$ and $\alpha_\sigma=h/\sigma$,
\begin{equation}\label{eq:lower-ratio-pde}
 \begin{gathered}
 \Psi_\sigma=2r\Psi_{rr}+\mathfrak d\Psi_r-2r\Psi_r^2+
                              \dot m\mathcal U,\\
 \mathfrak d(\sigma,r)=1-\frac{2r}{\sigma}
       +2\alpha_\sigma\sqrt r\tanh(\alpha_\sigma\sqrt r).
 \end{gathered}
\end{equation}
Set $P_\Psi=\Psi_r$ and $Z_\Psi=\Psi_{rr}$. Differentiation yields
\begin{align}
 (P_\Psi)_\sigma={}&2r(P_\Psi)_{rr}
 +(2+\mathfrak d-4rP_\Psi)(P_\Psi)_r
 +\mathfrak d_rP_\Psi-2P_\Psi^2+\dot m\mathcal U_r,
 \label{eq:lower-ratio-first}\\
 (Z_\Psi)_\sigma={}&2r(Z_\Psi)_{rr}
 +(4+\mathfrak d-4rP_\Psi)(Z_\Psi)_r
 +(2\mathfrak d_r-8P_\Psi)Z_\Psi\notag\\
 &-4rZ_\Psi^2+\mathfrak d_{rr}P_\Psi+\dot m\mathcal U_{rr}.
 \label{eq:lower-ratio-second}
\end{align}
For $x=\sqrt r>0$, the inequalities $C_x\le0$ and $B\ge xC$ imply
\begin{equation}\label{eq:squared-profile-concavity}
 \mathcal U_r=\frac B{2x}\ge0,\qquad
 \mathcal U_{rr}=\frac{xC-B}{4x^3}\le0.
\end{equation}
The remaining source term also has the required sign. With $z=\alpha_\sigma x$,
\begin{equation}\label{eq:lower-heat-coefficient}
 \mathfrak d_{rr}=\frac{\alpha_\sigma}{2x^3}
 [z\sech^2z-\tanh z-2z^2\sech^2z\tanh z]\le0.
\end{equation}
Indeed, $\tanh z-z\sech^2z$ vanishes at zero and has derivative
$2z\sech^2z\tanh z\ge0$. These signs extend to $r=0$ by smoothness.
When $h=0$, $\mathfrak d=1-2r/\sigma$ and $\mathfrak d_{rr}=0$.

On the first interval where $m$ is constant, $g$ is a constant multiple
of $k_h$, so $P_\Psi=Z_\Psi=0$. Choose $\sigma_0>0$ within this initial
interval. For any $T\in(\sigma_0,t)$, apply comparison on
$[\sigma_0,T]\times[0,\infty)$, with
$P_\Psi(\sigma_0,\cdot)=Z_\Psi(\sigma_0,\cdot)=0$.
The logarithmic derivative bounds above imply, uniformly in time,
\[
 P_\Psi(\sigma,r)=O(r^{-1/2}),\qquad
 Z_\Psi(\sigma,r)=O(r^{-1})
 \qquad (r\to\infty).
\]
Indeed, if $H(x)=\Psi(\sigma,x^2)$, then
\[
 P_\Psi=\frac{H_x}{2x},\qquad
 Z_\Psi=\frac{xH_{xx}-H_x}{4x^3},
\]
and $H_x,H_{xx}$ are uniformly bounded on the positive-time strip. 
At zero, the same bounds and evenness control their limits. Also
$\mathfrak d_r$ is bounded on this time interval.

Write $P=P_\Psi$ and $Z=Z_\Psi$. On
$[\sigma_0,T]\times[0,\infty)$, the reaction coefficients
\[
 c_P=\mathfrak d_r-2P,\qquad c_Z=2\mathfrak d_r-8P
\]
are bounded, and $P,Z\to0$ as $r\to\infty$, uniformly in time.
Choose $K$ larger than the absolute values of both reaction
coefficients on this strip.

To prove $P\ge0$, set
\[
 P_\varepsilon=P+\varepsilon e^{K(\sigma-\sigma_0)}.
\]
This function is positive initially and at spatial infinity. If it first
vanished at an interior point, then there
$(P_\varepsilon)_r=0$, $(P_\varepsilon)_{rr}\ge0$, and
$(P_\varepsilon)_\sigma\le0$. Since $P<0$ at that point and
$\dot m\mathcal U_r\ge0$, equation
\eqref{eq:lower-ratio-first} and the choice of $K$ instead give
$(P_\varepsilon)_\sigma>0$, a contradiction. At $r=0$, the
second-order term vanishes, the first-order coefficient equals $3$, and
a boundary minimum has $(P_\varepsilon)_r\ge0$, so the same contradiction
holds there. Letting $\varepsilon\downarrow0$ gives $P\ge0$.

Now set
\[
 Z_\varepsilon=Z-\varepsilon e^{K(\sigma-\sigma_0)}.
\]
Once $P\ge0$, the terms
$\mathfrak d_{rr}P$, $\dot m\mathcal U_{rr}$, and $-4rZ^2$ in
\eqref{eq:lower-ratio-second} are nonpositive. If $Z_\varepsilon$
first reached zero at an interior maximum, the equation and the choice
of $K$ would give $(Z_\varepsilon)_\sigma<0$, contradicting the
first-contact inequality $(Z_\varepsilon)_\sigma\ge0$. At $r=0$, the
first-order coefficient equals $5$ and a boundary maximum has
$(Z_\varepsilon)_r\le0$, so the same argument applies. Since
$Z_\varepsilon<0$ at spatial infinity, every possible first contact
occurs at finite $r$. Letting $\varepsilon\downarrow0$ gives $Z\le0$. 
The initial constant interval avoids the time-zero singularity. Letting
$T\uparrow t$ and using continuity gives the same signs at time $t$.

Finally use the smooth nondecreasing approximations from the bridge
argument, chosen constant near zero. Pointwise convergence of their
positive densities gives $\Psi_n(t,r)\to\Psi(t,r)$. Monotonicity and
concavity are preserved by this limit, and smoothness was already obtained
from \eqref{eq:lower-bridge-formula}. Multiplication by
$C(t,x)^2e^{m_*U(t,x)}$ and rescaling yield \eqref{eq:mc-rho-density}.
Since $\psi(r)\ge\psi(0)$, $C\le1$, and $|U_x|\le1$, the Gaussian
factor makes $\mathcal Z$ finite; it is positive because the density is
positive. The two signs for $\psi$ give
\eqref{eq:coefficient-s-bounds}.
\end{proof}

\section{Gaussian comparison and parameter bounds}\label{sec:gaussian-comparison}

The condition $N(t)\le0$ bounds the combinations
$\alpha=h/(\lambda t)$ and $d=1/(\lambda^2t)$.
We first compare the normalized second derivative of the Parisi PDE with
a Gaussian function. An explicit Gaussian estimate then yields bounds
for the original diffusion by comparison with reflected Brownian motion. These bounds
will control the $\theta$-dependent terms in the moment inequalities.

\subsection{Comparison with a Gaussian function}

Fix the PDE solution at time $t$, determined by a nondecreasing
coefficient on $[t,L]$. Write $B(x)=B(t,x)$ and $C(x)=C(t,x)$,
suppressing this fixed time. Put
\begin{equation}\label{eq:profile-normalization}
 {c_0}=C(0),\qquad \lambda^2=-\frac{C_{xx}(0)}{2{c_0}},\qquad
 A(v)=\frac{C(v/\lambda)}{c_0},\qquad
 u(v)=-\frac{A'(v)}{2A(v)},\qquad {I}(v)=\int_0^v A(y)\dd y.
\end{equation}
The normalization parameter $\lambda$ is positive. \zcref{lem:log-derivative,lem:normalized-estimates} give, for $v\ge0$,
\begin{equation}\label{eq:comparison-profile-estimates}
 \begin{gathered}
 u(0)=0,\quad u'(0)=1,\quad 0<u'\le1,\quad
 \tanh v\le u(v)\le v,\quad vu'(v)\le u(v),\\
 \lim_{v\to\infty}u(v)=1/\lambda<\infty.
 \end{gathered}
\end{equation}
In particular, $e^{-v^2}\le A(v)\le\sech^2v$ and ${I}(v)\le1$.

The functional below has the opposite sign from the expectation in the Brownian
comparison for $N$; this relation is used in \zcref{prop:parameter-bounds}.

\begin{lemma}\label{lem:auxiliary-functional}
For $d>0$ and $0<\alpha\le1$, set
$\mathfrak w(v)=e^{-dv^2/2}\cosh(\alpha v)$.
For a continuous function $P:[0,\infty)\to\R$ for which the
following integral converges absolutely, define
\[
 \mathcal E_\alpha[P]=\int_0^\infty
 \left[P(v)^2-d\left(\int_0^vP(y)\dd y\right)^2\right]
 \mathfrak w(v)\dd v.
\]
If $A_G(v)=e^{-v^2}$ and $\mathcal E_\alpha[A_G]<0$, then
$\mathcal E_\alpha[A]<0$ for every $A$ obtained by normalizing
$U_{xx}(t,\cdot)$ as in \eqref{eq:profile-normalization}.
\end{lemma}

\begin{proof}
Since the functional is negative at the Gaussian endpoint, it suffices to
show that its $\varepsilon$-derivative is strictly positive whenever it
vanishes along the following interpolation:
\begin{equation}\label{eq:mc-geometric-interpolation}
 A_\varepsilon(v)=A(v)^{1-\varepsilon}e^{-\varepsilon v^2},\qquad
 u_\varepsilon(v)=(1-\varepsilon)u(v)+\varepsilon v,
 \qquad 0\le\varepsilon\le1.
\end{equation}
Write $I_\varepsilon(v)=\int_0^v A_\varepsilon(y)\dd y$.
Every $u_\varepsilon$ has positive derivative, lies between $\tanh v$
and $v$, and satisfies
$vu_\varepsilon'-u_\varepsilon=(1-\varepsilon)(vu'-u)\le0$.
Also $A_\varepsilon\le A\le\sech^2v$. The function
\[
 \ell_A(v)=\partial_\varepsilon\log A_\varepsilon(v)
        =-v^2-\log A(v)
\]
is independent of $\varepsilon$ and satisfies
$-v^2\le\ell_A\le0$ and $\ell_A'=-2(v-u)\le0$.
The latter inequality is strict for sufficiently large $v$ by the
finite limit in \eqref{eq:comparison-profile-estimates}. Moreover,
\[
 |A_\varepsilon\ell_A|\le v^2\sech^2v,\qquad
 |\partial_\varepsilon {I_\varepsilon}(v)|
 \le\int_0^\infty y^2\sech^2y\dd y<\infty.
\]
These bounds justify differentiation of the functional on the entire
closed parameter interval, with one-sided derivatives at its endpoints.

Define
\[
 \mathcal K_\varepsilon(v)=A_\varepsilon(v)\mathfrak w(v)
       -d\int_v^\infty I_\varepsilon(y)\mathfrak w(y)\dd y.
\]
Fubini's theorem gives
\[
 \frac{\dd}{\dd\varepsilon}\mathcal E_\alpha[A_\varepsilon]
 =2\int_0^\infty A_\varepsilon\mathcal K_\varepsilon\ell_A\dd v.
\]
We determine the sign pattern of $\mathcal K_\varepsilon$ at a zero of
the functional. Fix $\varepsilon\in[0,1]$ and suppress its subscript. Put
$\mathcal S={I}/A$. Then
\[
 \mathcal S'=1+2u\mathcal S,\qquad \mathcal S(0)=0,\qquad
 \mathcal S''=2u'\mathcal S+2u\mathcal S'>0\quad(v>0).
\]
Convexity gives $\mathcal S(t)\le t\mathcal S(v)/v$ for $0\le t\le v$.
Since $u$ is strictly increasing, for $v>0$,
\begin{equation}\label{eq:comparison-convexity}
 \mathcal S(v)-v=2\int_0^v u(t)\mathcal S(t)\dd t
 <vu(v)\mathcal S(v),\qquad
 u'(v)(\mathcal S(v)-v)<u(v)^2\mathcal S(v).
\end{equation}
The second inequality uses $vu'\le u$.
To determine where $\mathcal K'$ changes sign, write
$b_\alpha(v)=\alpha\tanh(\alpha v)$ and define
\[
 {H_A}(v)=\frac{2u(v)-b_\alpha(v)}{\mathcal S(v)-v}.
\]
Fubini's theorem and differentiation give
\begin{equation}\label{eq:functional-kernel}
 \mathcal E_\alpha[A]=\int_0^\infty A{\mathcal K}\dd v,\qquad
 {\mathcal K}'=A{\mathfrak w}(\mathcal S-v)(d-{H_A}).
\end{equation}
Since $A\mathfrak w(\mathcal S-v)>0$ for $v>0$, it remains to show
that $H_A$ decreases from $\infty$ to zero.
Since $b_\alpha\le\tanh v\le u$ and $b_\alpha'\ge0$, differentiation
and \eqref{eq:comparison-convexity} give
\begin{align*}
 (\mathcal S-v)^2{H_A}'
 &=(2u'-b_\alpha')(\mathcal S-v)-2u\mathcal S(2u-b_\alpha)\\
 &<-2u\mathcal S(u-b_\alpha)-b_\alpha'(\mathcal S-v)\le0.
\end{align*}
Near zero, $\mathcal S-v\sim2v^3/3$ and $u(v)\sim v$, so ${H_A}\to\infty$.
At infinity, $u\le v$, $A\le\sech^2v$, and ${I}$ is bounded away from
zero, so ${H_A}\to0$.

There is therefore exactly one point at which $H_A=d$.
Thus ${\mathcal K}$ first strictly decreases and then strictly increases to zero;
its values on the final increasing interval are negative. If the
functional is zero, ${\mathcal K}(0)>0$ and ${\mathcal K}$ has exactly one zero $v_0$, with positive
sign before $v_0$ and negative sign after it.

At a zero of the functional, subtract $\ell_A(v_0)$ in the integrand. By
\eqref{eq:functional-kernel} the integral is unchanged. The resulting
integrand is nonnegative and strictly positive for all sufficiently large $v$.
Every zero of the functional therefore has strictly positive
$\varepsilon$-derivative, including a zero at $\varepsilon=0$ with its
right derivative. A
nonnegative initial value would become positive and could not reach
the negative value at $\varepsilon=1$. This proves the lemma.
\end{proof}

\subsection{An explicit Gaussian estimate}

We now find an explicit Gaussian tilt at which the comparison functional
is strictly negative.

\begin{lemma}\label{lem:gaussian-estimate}
For $d>0$, put
\begin{equation}\label{eq:mc-gaussian-threshold}
 \kappa=\frac{169}{625},\qquad
 \theta_*(d)=\min\left\{\kappa,\frac{13}{25}\sqrt d\right\},\qquad
 \alpha_*(d)=\sqrt{\theta_*(d)}.
\end{equation}
Then $0<\alpha_*(d)\le13/25<1$, and
$\mathcal E_{\alpha_*(d)}[A_G]<0$ for $A_G(v)=e^{-v^2}$.
\end{lemma}

\begin{proof}
Put $I_G(v)=\int_0^v A_G(y)\dd y$. The average of
$I_G^2-d^{-1}A_G^2$ under $\mathcal N(\alpha/d,1/d)$ has the
opposite sign from $\mathcal E_\alpha[A_G]$. We estimate this
average by first leaving the Gaussian parameters free.
For $\upsilon>0$ and $\eta\ge0$, let
$Y_G=\eta+\sqrt\upsilon Z_G$, where $Z_G$ is standard normal,
and put $L_G=1+2\upsilon$, $M_G=1+4\upsilon$. Define
\[
 {Q_G}(\eta)=\E {I_G}({Y_G})^2,\qquad {S_G}(\eta)=\E e^{-2{Y_G}^2}.
\]
Both functions also depend on $\upsilon$, which is suppressed in the notation.
A prime denotes differentiation in the mean $\eta$ at fixed variance;
$\partial_\upsilon Q_G(0)$ below differentiates the variance with the mean
held at zero. We use the normalization
\[
 \operatorname{erf}(x)=\frac2{\sqrt\pi}\int_0^x e^{-s^2}\dd s.
\]
Differentiating in the mean under the Gaussian expectation gives
\[
 Q_G'(\eta)=2\E\bigl[I_G(Y_G)e^{-Y_G^2}\bigr].
\]
Completing the square in the weighted Gaussian density yields
\[
 \E\bigl[I_G(Y_G)e^{-Y_G^2}\bigr]
 =L_G^{-1/2}e^{-\eta^2/L_G}
   \E I_G\left(\frac\eta{L_G}+\sqrt{\frac\upsilon{L_G}}Z_G\right).
\]
For $a\in\R$ and $b\ge0$, Gaussian averaging of $I_G$ gives
$\E I_G(a+bZ_G)=(\sqrt\pi/2)\operatorname{erf}
(a/\sqrt{1+2b^2})$; this follows by differentiating in $a$, completing
the square, and using that both sides vanish at $a=0$. Since
$1+2\upsilon/L_G=M_G/L_G$, we obtain
\[
 {S_G}(\eta)={M_G}^{-1/2}e^{-2\eta^2/{M_G}},\qquad
 {Q_G}'(\eta)=\sqrt{\pi/{L_G}}e^{-\eta^2/{L_G}}
       \operatorname{erf}(\eta/\sqrt{{L_G}{M_G}}).
\]
At mean zero, let ${\mathfrak I}=\E[{Y_G}{I_G}({Y_G})e^{-{Y_G}^2}]$. Gaussian integration by
parts and the heat equation give ${\mathfrak I}=\upsilon({S_G}(0)-2{\mathfrak I})$ and
$\partial_\upsilon {Q_G}(0)={S_G}(0)-2{\mathfrak I}$, respectively. Hence
$\partial_\upsilon {Q_G}(0)=1/({L_G}\sqrt {M_G})$, and integration from zero gives
${Q_G}(0)=(1/2)\arcsin(2\upsilon/{L_G})$.
Since $\operatorname{erf}(x)>(2x/\sqrt\pi)e^{-x^2}$ for $x>0$,
integration of the displayed formula for ${Q_G}'$ gives, strictly for $\eta>0$,
\[
 {Q_G}(\eta)>\frac12\arcsin(2\upsilon/{L_G})
       +\frac{\sqrt {M_G}}{2{L_G}}(1-e^{-2\eta^2/{M_G}}).
\]
Substitute $\upsilon=1/d$, $\eta=\alpha\upsilon$, and $z=2\upsilon/{L_G}=2/(d+2)$.
It follows that
\begin{equation}\label{eq:mc-Gaussian-deficit}
 2[{Q_G}(\eta)-\upsilon {S_G}(\eta)]>
 \arcsin z+\sqrt{1-z^2}
 -\frac{1+z-z^2}{\sqrt{1-z^2}}
  \exp\left(-\frac{\alpha^2z^2}{2(1-z^2)}\right).
\end{equation}
The left side has the opposite sign from $\mathcal E_\alpha[A_G]$.
It is therefore enough to show, at $\alpha=\alpha_*(d)$, that
\begin{equation}\label{eq:mc-exponential-target}
 \exp\left(\frac{\alpha^2z^2}{2(1-z^2)}\right)
 >\mathcal R(z)=
 \frac{1+z-z^2}{1-z^2+\sqrt{1-z^2}\arcsin z}.
\end{equation}
To estimate $\mathcal R(z)$ algebraically, introduce
\[
 r=\frac{z}{1+\sqrt{1-z^2}},\qquad q=\frac{1-r}{1+r},\qquad
 d=\frac{(1-r)^2}{r}=\frac{4q^2}{1-q^2}.
\]
The first substitution gives
\[
 z=\frac{2r}{1+r^2},\qquad
 \sqrt{1-z^2}=\frac{1-r^2}{1+r^2},\qquad \arcsin z=2\arctan r.
\]
Thus the square root becomes rational and the inverse sine becomes an
arctangent. The coordinate $r$ tends to zero as $d\to\infty$, whereas
$q$ tends to zero as $d\downarrow0$. We use a different polynomial estimate
in each regime. The two inequalities are proved in
\zcref{lem:arithmetic-certificate} below.

Suppose first that $d\ge\kappa$. Then $0<r<3/5$, since
$\kappa>4/15$, and $\alpha_*^2=\kappa$. Set
\[
 {N_U}(r)=r(105+55r^2),\qquad {H_U}(r)=105+90r^2+9r^4.
\]
The identity
\[
 \frac{\dd}{\dd r}\left(\arctan r-\frac{{N_U}(r)}{{H_U}(r)}\right)
 =\frac{576r^8}{(1+r^2){H_U}(r)^2}\ge0
\]
gives $\arctan r\ge {N_U}(r)/{H_U}(r)$. Put
\[
 A_U=(1-r^2)^2{H_U}+2(1-r^4){N_U},\qquad
 B_U=[(1-r^2)^2+2r(1+r^2)]{H_U},
\]
\[
 {D_U}=(1-r^2)^2,\qquad {Z_U}=\frac{338}{625}r^2,
\]
so that $y=Z_U/D_U=\alpha_*^2z^2/[2(1-z^2)]$ is the exponent in
\eqref{eq:mc-exponential-target}. The arctangent bound gives
$\mathcal R(z)\le B_U/A_U$. We compare this upper bound with the
rational lower bound $(6+2y)/(6-4y+y^2)$ for $e^y$, proved below.
Clearing the positive denominators produces a numerator divisible by
$r^2$. Define the resulting degree-$13$ polynomial by
\begin{equation}\label{eq:mc-uniform-poly}
 p_U(r)=\frac{A_U(6{D_U}^2+2{Z_U}{D_U})-B_U(6{D_U}^2-4{Z_U}{D_U}+{Z_U}^2)}{r^2}.
\end{equation}
The quotient is continued polynomially at zero. For $y\ge0$,
\begin{equation}\label{eq:exponential-rational-bound}
 e^y\ge\frac{6+2y}{6-4y+y^2}.
\end{equation}
Indeed, the denominator is $(y-2)^2+2>0$, and
\[
 \frac{\dd}{\dd y}\left[y-\log\frac{6+2y}{6-4y+y^2}\right]
 =\frac{y^3}{(y+3)(y^2-4y+6)}\ge0.
\]
\zcref{lem:arithmetic-certificate} gives $p_U>0$ on $[0,3/5]$.
Thus
\[
 e^{{Z_U}/{D_U}}\ge\frac{6{D_U}^2+2{Z_U}{D_U}}{6{D_U}^2-4{Z_U}{D_U}+{Z_U}^2}>\frac{B_U}{A_U}
 \ge\frac{(1-r^2)^2+2r(1+r^2)}
 {(1-r^2)^2+2(1-r^4)\arctan r}
 =\mathcal R(z).
\]
This proves \eqref{eq:mc-exponential-target} in the first case.

Suppose now that $d\le\kappa$. Then $0<q<13/50$, since
$\kappa<676/2331$, and $\alpha_*^2=(13/25)\sqrt d$.
Evaluating the same rational lower bound at $r=1$ gives
$\arctan 1\ge40/51$. 
As $\arctan q\le q$ and $\arctan r=\arctan1-\arctan q$, the identity
\[
 \mathcal R(z)=
 \frac{1+4q^2-q^4}
 {4q[q+(1+q^2)(\arctan1-\arctan q)]}
\]
implies $\mathcal R(z)\le B_C(q)/(4qA_C(q))$, where
\[
 A_C(q)=\frac{40}{51}(1+q^2)-q^3>0,\qquad
 B_C(q)=1+4q^2-q^4.
\]
On $0\le q\le13/50$,
\[
 \sqrt{1-q^2}\ge1-\frac{51}{100}q^2.
\]
Indeed the right side is positive and its square is at most $1-q^2$,
since $(51/100)^2(13/50)^2<1/50$. Put
\[
 K(q)=\frac{13}{100}(1-q^2)\left(1-\frac{51}{100}q^2\right)
\]
so that, for $q>0$,
\[
 \frac{K(q)}q\le\frac{13(1-q^2)^{3/2}}{100q}
 =\frac{\alpha_*^2z^2}{2(1-z^2)}.
\]
We compare the cubic Taylor lower bound for $e^{K(q)/q}$ with
$B_C(q)/(4qA_C(q))$. Multiplication of their difference by the positive
factor $4q^3A_C(q)$ gives the degree-$15$ polynomial
\begin{equation}\label{eq:gaussian-d-polynomial}
 p_C(q)=4A_C(q)
 \left(q^3+K(q)q^2+\frac{K(q)^2q}{2}+\frac{K(q)^3}{6}\right)
 -B_C(q)q^2.
\end{equation}
\zcref{lem:arithmetic-certificate} gives $p_C>0$ on $[0,13/50]$.
The bound $e^y\ge1+y+y^2/2+y^3/6$ for $y\ge0$ and the square-root
estimate give
\[
 \exp\left(\frac{13(1-q^2)^{3/2}}{100q}\right)
 \ge e^{K(q)/q}>\frac{B_C(q)}{4qA_C(q)}\ge\mathcal R(z).
\]
This proves \eqref{eq:mc-exponential-target} in the second case
and completes the proof.
\end{proof}

We transfer this Gaussian sign to the original diffusion and obtain the
parameter bounds used in the moment argument.

\begin{proposition}\label{prop:parameter-bounds}
Let $m:[0,L]\to[0,1]$ be nondecreasing, let $0<t\le L$, and let $h>0$.
Let $U$ and $Y$ solve \eqref{eq:parisi-time}--\eqref{eq:optimal-diffusion},
and use the normalization in \eqref{eq:profile-normalization}. Define
\[
 N=Q(t)-tS(t),\qquad \alpha=\frac h{\lambda t},\qquad
 d=\frac1{\lambda^2t}.
\]
Then
\begin{equation}\label{eq:parameters-under-N}
 N\le0\quad\Longrightarrow\quad
 \alpha^2<\frac{169}{625},\qquad
 \alpha<\frac{13}{25},\qquad
 d>\frac{625}{169}\alpha^4.
\end{equation}
\end{proposition}

\begin{proof}
The even function
$\mathfrak n(x)=B(t,x)^2-tC(t,x)^2$ is strictly increasing for $x>0$, since
\[
 \mathfrak n'(x)=2B(t,x)C(t,x)-2tC(t,x)C_x(t,x)>0.
\]
The drift in \eqref{eq:mc-lower-sde} is odd and nonnegative on the
positive half-line. Define, with $\operatorname{sgn}(0)=0$,
\[
 \widetilde W_r=\int_0^r\operatorname{sgn}(Y_s)\dd W_s.
\]
Since $\dd\langle Y\rangle_s=\dd s$, the occupation-time formula gives
$\int_0^r\mathbf1_{\{Y_s=0\}}\dd s=0$. Hence
$\langle\widetilde W\rangle_r=r$, and L\'evy's characterization shows that
$\widetilde W$ is Brownian. With the usual symmetric-local-time convention,
Tanaka's formula gives
\[
 \dd|Y_r|=m(r)B(r,|Y_r|)\dd r+\dd\widetilde W_r+\dd L_r^0(Y).
\]
Thus $|Y|$ is a reflected one-dimensional diffusion with this outward drift.
Couple it with a reflected Brownian motion, started at $h$, using
$\widetilde W$. Write the coupled reflected processes as $R$ and $R^0$,
with regulators $L^R,L^0$. Their difference has finite variation. On
$\{R<R^0\}$ one has $R^0>0$, so $\dd L^0=0$, while both the outward drift
of $R$ and $\dd L^R$ are nonnegative. The chain rule for
$(R^0-R)_+$ therefore shows that this positive part cannot increase from
zero. Hence $R\ge R^0$ pathwise.  The bounded
Lipschitz drift guarantees existence and pathwise uniqueness, so this
comparison applies throughout the interval.  Consequently
\begin{equation}\label{eq:mc-brownian-comparison}
 N\ge\E\mathfrak n(h+\sqrt t Z_G),
\end{equation}
where $Z_G$ is standard normal.

Under $v=\lambda x$, the Brownian variable has variance $1/d$ and mean
$\alpha/d$, and
$B(t,x)=({c_0}/\lambda){I}(\lambda x)$,
$C(t,x)={c_0}A(\lambda x)$. The density of the absolute value of the
rescaled Gaussian is a positive constant times
$e^{-dv^2/2}\cosh(\alpha v)$.  Therefore the Brownian expectation in
\eqref{eq:mc-brownian-comparison} has the opposite sign from
$\mathcal E_\alpha[A]$.

Throughout the following comparison, $t$ and the entire coefficient $m$
are fixed. Thus $U,\lambda,d$ are fixed; only the Gaussian starting point
varies, with $h_i=\lambda t\alpha_i$.
Take $\alpha_* =\alpha_*(d)$ from
\eqref{eq:mc-gaussian-threshold}. By \zcref{lem:gaussian-estimate},
$0<\alpha_*\le13/25<1$ and $\mathcal E_{\alpha_*}[A_G]<0$.
Thus \zcref{lem:auxiliary-functional} applies with this $d$, this
$\alpha_*$, and the fixed normalized profile $A$, giving
$\mathcal E_{\alpha_*}[A]<0$.
Hence the Brownian expectation is strictly positive at the field
corresponding to $\alpha_*$. To compare shifts
$\alpha_2>\alpha_1\ge0$, let $\pi_i$ be the probability measure on
$[0,\infty)$ with density proportional to
$e^{-dv^2/2}\cosh(\alpha_i v)$. Put
\[
 f(v)=\mathfrak n(v/\lambda),\qquad
 r(v)=\frac{\cosh(\alpha_2v)}{\cosh(\alpha_1v)}.
\]
Both $f$ and $r$ are strictly increasing on $(0,\infty)$, and
$\pi_2(\dd v)=r(v)\pi_1(\dd v)/\E_{\pi_1}r$. Therefore
\[
 \E_{\pi_2}f-\E_{\pi_1}f
 =\frac{\Cov_{\pi_1}(f,r)}{\E_{\pi_1}r}>0.
\]
The strict inequality follows because $f$ and $r$ are strictly increasing
and $\pi_1$ has positive density. Thus the Brownian
expectation, and hence $N$, is positive whenever $\alpha\ge\alpha_*$.
Consequently $N\le0$ implies
\[
 \alpha^2<\theta_*(d)
 =\min\left\{\frac{169}{625},\frac{13}{25}\sqrt d\right\},
\]
which gives all the asserted inequalities.
\end{proof}

\section{Moment inequalities for the \texorpdfstring{$C^2$}{C2}-weighted measure}
\label{sec:weighted-moments}

Let $m:[0,L]\to[0,1]$ be nondecreasing, and let $U,Y$ satisfy
\eqref{eq:parisi-time}--\eqref{eq:optimal-diffusion}, with $Y_0=h\ge0$.
Fix $0<t\le L$ and $m_0\in(0,1]$ such that
\[
 m\le m_0\quad\text{almost everywhere on }[0,t),\qquad
 m\ge m_0\quad\text{almost everywhere on }(t,L].
\]
Every term in
$J(t;m_0)$ contains a factor $C^2$ after the normalization of
\zcref{sec:profiles}. Absorbing that factor into the law converts
the condition $F(t)=2m_0G(t)$ into one moment identity and allows the
remaining terms to be compared by integration by parts.

Apply \zcref{lem:normalized-estimates} at reverse time $\tau=L-t$,
using the fixed value $m_0$, and use the normalized functions defined
there. Primes on the normalized spatial functions denote derivatives with
respect to $v$; primes on $\psi$ below denote derivatives with respect
to its argument $r$. Thus $a,j$ are even, $u$ is odd, and, for $v\ge0$,
\begin{equation}\label{eq:moment-profile-inequalities}
\begin{gathered}
 a(0)=\mu,\qquad u(0)=0,\qquad u'(0)=1,\\
 a'=-2au,\qquad u'=a+j,\qquad -2uj\le j'\le uj,\\
 j\ge ka,\qquad 0<u'\le1,\qquad k=\frac{1-\mu}{\mu},
\end{gathered}
\end{equation}
where $0<\mu\le1$. Write $b(v)=m_0B(t,v/\lambda)/\lambda$, so that
$b'=a$ and $b(0)=0$. Set
\[
 w=u^2,\qquad {\mathcal X}=\mu-a,\qquad
 \mathcal T=1-a+2w-j,\qquad p=a-2w.
\]
Here $\mathcal X$ measures the decrease of $a$ from its value at zero, and
$p$ encodes the moment constraint. The function
$\mathcal T=1-u'+2w$ is chosen so that $j+a-2w=1-\mathcal T$.
The probability measure $\rho$ is the law of $\lambda|Y_t|$, reweighted
by $C(t,Y_t)^2$ and normalized. Unless otherwise indicated, all
expectations in this section are under $\rho$.
All the estimates below assume the moment identity
\begin{equation}\label{eq:moment-constraint}
 \E_\rho p=0.
\end{equation}
By \zcref{lem:lower-density}, the density has the form
\begin{equation}\label{eq:moment-density}
 \rho(\dd v)=\frac1{\mathcal Z}\omega(v)e^{-dv^2/2-\psi(v^2)}
                  \cosh(\alpha v)\dd v,\qquad v\ge0,
 \qquad (\log\omega)'=-4u+b.
\end{equation}
Here $d>0$, $\psi'\ge0$, $\psi''\le0$, and
\[
 s(v)=d+2\psi'(v^2),\qquad d\le s(v)\le s(0)<\infty,
 \qquad s'(v)=4v\psi''(v^2)\le0\quad(v\ge0).
\]
We also impose the parameter bounds
\begin{equation}\label{eq:moment-parameter-bounds}
 \kappa=\frac{169}{625},\qquad
 0\le\theta=\alpha^2\le\kappa,
 \qquad d>\frac{\theta^2}{\kappa}.
\end{equation}
For $h>0$, these follow from \zcref{prop:parameter-bounds} when
$N(t)\le0$. For $h=0$, one has $\alpha=0$ and $d>0$, so the bounds hold
automatically. In either case $\alpha<1$.

Throughout the remainder of this section, fix $t$ and $m_0$ as above: $m_0$ satisfies the two
one-sided bounds on $m$, the normalized functions are those of
\zcref{lem:normalized-estimates}, and $\rho$ is the $C^2$-weighted
law of the diffusion determined by this same coefficient $m$.
Together with \eqref{eq:moment-constraint} and
\eqref{eq:moment-parameter-bounds}, these are the standing assumptions
invoked below and in \zcref{sec:positivity}.

At the fixed time $t$, differentiation in $v=\lambda x$ gives
\[
 C_x=-2\lambda Cu,\qquad
 C_{xx}=-2\lambda^2C(j+a-2u^2),
\]
where $C,C_x,C_{xx}$ are evaluated at $(t,v/\lambda)$ and $a,j,u$ at $v$.
Set $M=\E_\rho a>0$. Define the normalized integrand in $J$ by
\[
 {\mathcal J}=2(j+a-2w)^2+12a(a-2w)-9a^2,
\]
and put $O=\E_\rho\mathcal J/M^2$. The definition of $\rho$ gives
\begin{equation}\label{eq:weighted-moment-scaling}
 F(t)-2m_0G(t)=-2\lambda^2S(t)\E_\rho p,\qquad
 J(t;m_0)=2\lambda^4S(t)M^2O.
\end{equation}
Thus the moment identity is equivalent to $F(t)=2m_0G(t)$, and proving
$O>0$ proves $J(t;m_0)>0$. Division by $M^2$ puts the expression on the
same scale as the normalized variances and mixed moments.

\subsection{Consequences of the differential inequalities}

The preceding inequalities give $0<a\le\mu$, $0\le j\le1$, and
$\tanh v\le u\le v$ on $v\ge0$. Moreover, ${\mathcal X},w$ are increasing, and
\begin{equation}\label{eq:basic-pointwise-comparisons}
 w\ge\frac {\mathcal X}\mu=1-\frac a\mu,\qquad a\ge\mu(1-w).
\end{equation}
Since $j\ge ka$, one also has
$w'=2u(a+j)\ge2ua/\mu={\mathcal X}'/\mu$.

Direct differentiation gives
\[
 \mathcal T'=6au+4uj-j'.
\]
Consequently
\begin{equation}\label{eq:profile-increment-comparisons}
 \mathcal T'\ge\frac32({\mathcal X}'+w'),\qquad
 w'\ge\frac{{\mathcal X}'}\mu.
\end{equation}
Notice that $\mathcal T(0)=0$, since $j(0)=1-\mu$.

Define the normalized means, variances, and mixed moments by
\begin{equation}\label{eq:moment-definitions}
\begin{gathered}
 g=\frac\mu M-1,\qquad {\xi}=\frac1M=\frac{1+g}{\mu},
 \qquad y=\frac{\E j}{M},\\
 {V_a}=\frac{\Var(a)}{M^2},\qquad
 \chi=-\frac{\Cov(a,w)}{M^2},\qquad
 {V_T}=\frac{\Var(\mathcal T)}{M^2},\qquad
 {V_w}=\frac{\Var(w)}{M^2},\\
 {R_j}=\frac{\E(jw)}{M^2},\qquad
 {R_a}=\frac{\E(aw)}{M^2}.
\end{gathered}
\end{equation}
Since $j+a-2w=1-\mathcal T$, one has
${\mathcal J}=2(1-\mathcal T)^2+3a^2-24aw$ and $\E(1-\mathcal T)=My$.
Consequently
\begin{equation}\label{eq:normalized-J}
 O=\frac{\E {\mathcal J}}{M^2}=2{V_T}+3{V_a}+24\chi+2y^2-9.
\end{equation}
This identity uses only the moment identity \eqref{eq:moment-constraint}.

The following bounds on means, variances, and covariance will show that
the nonnegative terms in \eqref{eq:normalized-J} outweigh the constant
$-9$.

The equivalent formulas
\[
 g=\frac{\E_\rho\mathcal X}{M},\qquad
 \chi=\frac{\Cov_\rho(\mathcal X,w)}{M^2}
\]
identify $g$ with the mean decrease of $a$, relative to $M$, and $\chi$
with the covariance of two increasing functions. In particular, both are
nonnegative. The notation will be used as follows.
\begin{center}
\begin{tabular}{ll}
\toprule
Quantities & Role \\
\midrule
$M,\xi$ & Mean scale and its reciprocal \\
$g,y$ & Relative mean decrease of $a$ and normalized mean of $j$ \\
$V_a,V_w,V_T$ & Variance terms controlled by comparison and integration by parts \\
$\chi$ & Covariance retained with a positive coefficient in the final bound \\
$R_a,R_j$ & Mixed moments arising from derivative expectations \\
\bottomrule
\end{tabular}
\end{center}

The moment identity gives
\begin{equation}\label{eq:moment-relations}
 \E w=\frac M2,\qquad {R_a}=\frac12-\chi.
\end{equation}
For the variance expansions below, the definition
$\mathcal T=1-a+2w-j$ and $\E j=My$ also give
\begin{equation}\label{eq:covariance-T-w}
 \frac{\Cov_\rho(\mathcal T,w)}{M^2}
 =\chi+2V_w-R_j+\frac y2.
\end{equation}

The variables ${V_a},{V_T},{V_w},{R_j},\chi,{R_a}$ are nonnegative.
For $\chi$, use that $a$ decreases and $w$ increases. More generally,
let $f_1,f_2\in L^2(\rho)$ be real-valued, and let $V_1,V_2$ be
independent with law $\rho$. Then
\[
 \Cov_\rho(f_1,f_2)=\frac12\E\bigl[
 (f_1(V_1)-f_1(V_2))(f_2(V_1)-f_2(V_2))\bigr],
\]
where the expectation on the right is under $\rho\otimes\rho$.
The functions ${\mathcal X},w,w-{\mathcal X}/\mu$ are increasing. Applying the covariance
identity first to ${\mathcal X},w-{\mathcal X}/\mu$ and then to $w,w-{\mathcal X}/\mu$ gives
\begin{equation}\label{eq:variance-comparisons}
 {V_a}\le\mu\chi,\qquad \chi\le\mu {V_w}.
\end{equation}

We also need the mean-dependent comparison
\begin{equation}\label{eq:convex-covariance}
 {V_a}\le2g\chi.
\end{equation}
Since $\mathcal X'=2au>0$ on $v>0$, $w$ can be regarded as a
function of $\mathcal X$. Its derivatives satisfy
\[
 \frac{\dd w}{\dd {\mathcal X}}=1+\frac ja,
 \qquad
 \left(\frac ja\right)'=\frac{j'+2uj}{a}\ge0.
\]
Thus $w({\mathcal X})$ is convex and vanishes at zero, so $w({\mathcal X})/{\mathcal X}$ is
nondecreasing for ${\mathcal X}>0$. If $\E_\rho\mathcal X>0$, define
expectation under the $\mathcal X$-weighted probability measure by
\[
 \E_{\mathcal X}f=\frac{\E_\rho(\mathcal X f)}{\E_\rho\mathcal X}.
\]
The two functions $\mathcal X$ and $w(\mathcal X)/\mathcal X$ are increasing,
so
\[
 \Cov_{\mathcal X}\left(\mathcal X,\frac{w(\mathcal X)}{\mathcal X}\right)\ge0.
\]
Multiplying by $(\E_\rho\mathcal X)^2$ gives
\[
 (\E {\mathcal X})\E({\mathcal X}w)\ge(\E {\mathcal X}^2)(\E w).
\]
Subtracting $(\E {\mathcal X})^2\E w$ from both sides shows
\[
 (\E {\mathcal X})\Cov({\mathcal X},w)\ge(\E w)\Var({\mathcal X}).
\]
Since $\E {\mathcal X}=gM$ and $\E w=M/2$, this is
\eqref{eq:convex-covariance}. If $\E {\mathcal X}=0$, then ${\mathcal X}=0$ almost surely
and the same conclusion is immediate.

The pointwise comparisons and their integrated versions give
\begin{equation}\label{eq:global-mean-bounds}
 \frac{2\mu}{\mu+2}\le M\le\mu,\qquad
 0\le g\le\frac\mu2,\qquad
 y\ge k,\qquad
 y\le\left(1-\frac{3\mu}{2}\right){\xi}+\frac34.
\end{equation}
For the lower bound on $M$, average $a\ge\mu(1-w)$ and use the
moment identity. The bound $y\ge k$ follows from $j\ge ka$. Finally,
integrate the first inequality of
\eqref{eq:profile-increment-comparisons} to obtain
$\mathcal T\ge(3/2)({\mathcal X}+w)$. Its expectation is
\[
 1-My=\E\mathcal T\ge\frac{3\mu}{2}-\frac{3M}{4},
\]
which proves the upper bound on $y$.

\subsection[A quadratic bound for j - ka]{A quadratic bound for $j-ka$}

We need an upper bound for $\E_\rho(j-ka)^2$: subtracting this
moment from its upper bound will supply a nonnegative term in the
final linear combination. Put
$K=j-ka$ and $e=w-1+a/\mu=w-{\mathcal X}/\mu$.
Then $K,e\ge0$, both vanish at zero, and
\[
 e'=2uK,\qquad K'\le u(K+3ka).
\]
Moreover
\[
 K+3ka=j+2ka=u'+(2k-1)a
 \le\max\{1,3(1-\mu)\}.
\]
To verify the last bound, if $2k-1\le0$ use $u'\le1$; otherwise
use $a\le\mu$, which gives $1+(2k-1)\mu=3(1-\mu)$.
It follows that
\[
 (K^2)'\le\max\{1,3(1-\mu)\}e'.
\]
For $11/20\le\mu\le29/40$, the affine function
$49/20-2\mu$ majorizes both $1$ and $3(1-\mu)$.
Integration and then averaging give the explicit bound
\begin{equation}\label{eq:quadratic-bound-j-ka}
 K^2\le\left(\frac{49}{20}-2\mu\right)e,
 \qquad
 \frac{\E K^2}{M^2}\le Q_*=
 \left(\frac{49}{20}-2\mu\right)
 \frac{(\mu-2g)(1+g)}{2\mu^2}.
\end{equation}
Here we used
$\E e/M^2=(\mu-2g)(1+g)/(2\mu^2)$, which follows directly from
$\E w=M/2$ and $M=\mu/(1+g)$.

\subsection[Integration by parts under rho]{Integration by parts under $\rho$}

To bound the variances in $O$ from below, we will express derivative
expectations as covariances with a centered function determined by the
density of $\rho$. Cauchy--Schwarz then relates each variance to the
second moment of that function. We bound this second moment from
above using the term $vs'(v)\le0$ in the second derivative of the
negative log-density.

We first record a bound on the drift contribution $bu$.
The bound $b^2\le\mu(\mu-a)$ in
\eqref{eq:normalized-magnetization-bounds} and the nonnegative square
$(b/\sqrt\mu-\sqrt\mu u)^2$ give
\[
 2bu\le\frac{b^2}{\mu}+\mu u^2\le\mu-a+\mu w.
\]
Taking expectations, using $\mu-M=gM$ and $\E w=M/2$, yields
\begin{equation}\label{eq:bound-bu}
 \frac{\E(bu)}M\le\frac\mu4+\frac g2.
\end{equation}

To control the contribution of the starting point, we replace the
half-line factor $\cosh(\alpha v)$ by $e^{\alpha v}$ on $\R$.
Let $\widehat\E$ denote expectation under the probability measure
with density proportional to
$\omega(v)e^{-dv^2/2-\psi(v^2)+\alpha v}$, with the even extension
of $\omega$. This law has the same expectations of even functions
as $\rho$. More precisely, pairing the two half-lines gives, for
an integrable even or odd function $f$,
\begin{equation}\label{eq:parity-expectations}
 \widehat\E f=
 \begin{cases}
  \E_\rho f(v),&f\text{ even},\\
  \E_\rho[f(v)\tanh(\alpha v)],&f\text{ odd},
 \end{cases}
\end{equation}
where $f$ on the right is restricted to $v\ge0$.
Integrating the logarithmic derivative of the density on $\mathbb R$ gives
\[
 \alpha=\widehat\E[4u-b+s(v)v].
\]
For $v\ge0$,
\[
 4u-b+s(v)v-(d+4-\mu)u
 =\mu u-b+s(v)(v-u)+(s(v)-d)u\ge0.
\]
This difference is odd and nonnegative for $v\ge0$. Since $\alpha\ge0$,
\eqref{eq:parity-expectations} therefore gives
\begin{equation}\label{eq:tilted-mean-bound}
 0\le\widehat\E u\le\frac{\alpha}{d+4-\mu}.
\end{equation}
Define the weighted average
\begin{equation}\label{eq:weighted-average-s}
 \Sigma=\frac{2\E_\rho[s(v)w]}M\in[d,s(0)],
\end{equation}
where the bounds use the moment identity $\E_\rho w=M/2$.

Let $\mathcal V$ be the negative logarithm of the density of $\rho$,
up to an additive constant, and put ${\mathcal Y}=u\mathcal V'-u'$. This
choice satisfies
\[
 \mathcal Y e^{-\mathcal V}=-\bigl(u e^{-\mathcal V}\bigr)',
\]
while
\begin{align*}
 \mathcal V'&=4u-b+s(v)v-\alpha\tanh(\alpha v),\\
 \mathcal V''&=4u'-a+s(v)+vs'(v)-\theta\sech^2(\alpha v).
\end{align*}
For $f\in C^1([0,\infty))$ with $f,f'$ of at most polynomial growth,
integration by parts gives the first two identities below. The same
calculation applies to $f=\mathcal Y$ by the integrability argument that
follows and yields the third identity:
\begin{equation}\label{eq:weighted-ibp}
 \E_\rho {\mathcal Y}=0,\qquad
 \E_\rho(f{\mathcal Y})=\E_\rho(f'u),\qquad
 \E_\rho {\mathcal Y}^2=\E_\rho\bigl[u'^2+w\mathcal V''\bigr].
\end{equation}
For the last identity, apply the second to $f={\mathcal Y}$ and use
$\E_\rho[u u'\mathcal V']=\E_\rho[(u u')']$.
The boundary term at zero vanishes because $u(0)=0$. The bound
$b\le\mu u$ gives $(\log\omega)'=-4u+b\le0$. Also
$e^{-\psi(v^2)}$ is nonincreasing, so the density has Gaussian decay
at infinity. The test functions and ${\mathcal Y}$ have at most polynomial growth,
since $s$ is bounded. The term involving $s'$ is integrable: $-s'$ is
nonnegative with integral at most $s(0)-d$, while a polynomial times
the density is bounded. These facts justify all the integrations by
parts, including the one with $f={\mathcal Y}$.

Define
\begin{equation}\label{eq:quantity-H}
 \mathcal H=\frac{\E_\rho {\mathcal Y}^2}{M^2}
     =\frac{\E_\rho\bigl[u'^2+w\mathcal V''\bigr]}{M^2}\ge0.
\end{equation}
For such $f$ and every real $\eta$, completing a square gives
\begin{equation}\label{eq:variance-inequality}
 \frac{\Var_\rho(f)}{M^2}
 -2\eta\frac{\E_\rho(f'u)}{M^2}+\eta^2\mathcal H
 =\frac{\E_\rho[(f-\E_\rho f-\eta {\mathcal Y})^2]}{M^2}\ge0.
\end{equation}
The monotonicity of $s$ gives
\begin{equation}\label{eq:H-first-upper-bound}
 \mathcal H\le\frac{\E_\rho\bigl[u'^2+w[4u'-a+s(v)
                    -\theta\sech^2(\alpha v)]\bigr]}{M^2},
\end{equation}
because $vs'(v)\le0$. Since $\alpha\le1$ and $u\ge\tanh v$,
$\sech^2(\alpha v)\ge1-w$. Expansion therefore gives
\begin{equation}\label{eq:H-moment-upper-bound}
 \mathcal H\le\overline H_\theta
 ={V_T}+y^2+\frac52-7\chi+(\theta-4){V_w}+8{R_j}
 +\frac{(\Sigma-\theta){\xi}}{2}+\frac\theta4.
\end{equation}
The terms independent of $s$ and $\theta$ are
\[
 \frac{\E_\rho[(a+j)^2+3aw+4jw]}{M^2}
 ={V_T}+y^2+\frac52-7\chi-4{V_w}+8{R_j}.
\]
The term containing $s$ contributes $\Sigma {\xi}/2$, and the term
containing $\theta$ contributes at most $-\theta {\xi}/2+\theta {V_w}+\theta/4$.
For any fixed $t_*\ge\theta$ we may use the larger upper bound
\begin{equation}\label{eq:H-enlarged-bound}
 H_{t_*}={V_T}+y^2+\frac52-7\chi+(t_*-4){V_w}+8{R_j}
 +\frac{(\Sigma-\theta){\xi}}{2}+\frac\theta4.
\end{equation}
Only the coefficient of ${V_w}$ is increased. The terms involving
$\Sigma$ and $\theta$ are otherwise unchanged.

\subsection[A lower bound for y and five moment inequalities]{A lower bound for $y$ and five moment inequalities}

We first isolate the mean estimate used in all three ranges of $\mu$
and in the positive- and zero-field bounds for the upper atom.
Set $\gamma=4-\mu$ and $a_0=1-\mu/4-g/2$.

\begin{lemma}\label{lem:general-mean-bound}
Under the standing assumptions of \zcref{sec:weighted-moments}, for
every $0<\mu\le1$ one has
\begin{equation}\label{eq:lower-bound-y}
 y\ge a_0+\frac\Sigma2-\frac{\theta {\xi}}{d+{\gamma}},
 \qquad \Sigma\ge d>\frac{\theta^2}{\kappa}.
\end{equation}
\end{lemma}

\begin{proof}
The logarithmic derivative of the full-line density is
$-4u+b-s(v)v+\alpha$. Integrating the derivative of $u$ against
that density therefore gives
\[
 \widehat\E u'=\widehat\E\bigl[u(4u-b+s(v)v-\alpha)\bigr].
\]
Use $u'=a+j$, $\E_\rho w=M/2$, and the parity identity
\eqref{eq:parity-expectations}. The even terms have the same
expectations under the two laws, so
\[
 M(1+y)=2M-\E_\rho(bu)+\E_\rho[s(v)vu]-\alpha\widehat\E u.
\]
Use $vu\ge w$, \eqref{eq:tilted-mean-bound}, and
\eqref{eq:bound-bu}, and divide by $M$, to obtain the first inequality
in \eqref{eq:lower-bound-y}. The bounds on $\Sigma$ and $d$ follow from
\eqref{eq:weighted-average-s} and \eqref{eq:moment-parameter-bounds}.
\end{proof}

For the following five quantities only, assume
$11/20\le\mu\le29/40$, so that \eqref{eq:quadratic-bound-j-ka} applies.
The first combines that bound with a variance:
\begin{equation}\label{eq:nonnegative-square-combination}
 R_\square=\frac{\Var(5j+a-w)}{M^2}
   +20\mu\left(Q_*-\frac{\E(j-ka)^2}{M^2}\right).
\end{equation}
Here $Q_*$ is defined in \eqref{eq:quadratic-bound-j-ka}. Expansion using the moment
identity gives
\begin{equation}\label{eq:combined-square-expansion}
\begin{aligned}
 R_\square={}&(25-20\mu){V_T}+\left(16-\frac{20}{\mu}\right){V_a}
                  +(80\mu-99){V_w}+(80\mu-98)\chi\\
 &+(90-80\mu){R_j}-20\mu y^2-5y+20\mu(Q_*-k^2).
\end{aligned}
\end{equation}
Define the remaining four quantities by
\begin{equation}\label{eq:nonnegative-moments}
\begin{aligned}
 R_C&=-{V_a}+2g\chi,\\
 R_1&=36{V_a}+49{V_w}+84\chi-\frac{26}{5}(26{R_a}+14{R_j})
       +\frac{169}{25}H_\kappa,\\
 R_2&=100{V_T}-279{V_w}-180\chi+180{R_j}-90y
       -7(42{R_a}+12{R_j})+\frac{49}{4}H_\kappa,\\
 R_{1/3}&={V_a}-\frac23\chi+\frac19{V_w}.
\end{aligned}
\end{equation}
These five quantities give the nonnegative terms used in the final linear
combination.

\begin{lemma}\label{lem:five-nonnegative-moments}
Under the standing assumptions of \zcref{sec:weighted-moments}, suppose that
$11/20\le\mu\le29/40$. Then the quantities defined in
\eqref{eq:nonnegative-square-combination} and
\eqref{eq:nonnegative-moments} satisfy
\[
 R_\square\ge0,\qquad R_C\ge0,\qquad R_1\ge0,\qquad
 R_2\ge0,\qquad R_{1/3}\ge0.
\]
\end{lemma}

\begin{proof}
The sign of $R_\square$ follows from \eqref{eq:quadratic-bound-j-ka} and
nonnegativity of variance. The inequality $R_C\ge0$ is
\eqref{eq:convex-covariance}, and $R_{1/3}=\Var(a+w/3)/M^2$.
The two inequalities obtained by integration by parts use the upper bound
$H_\kappa$ from \eqref{eq:H-enlarged-bound}.

For $R_1$, apply~\eqref{eq:variance-inequality} with
$f=7w-6a$, $\eta=13/5$, and use this upper bound for $\mathcal H$. Its variance is
$M^2(36{V_a}+49{V_w}+84\chi)$ and the corresponding term in
\eqref{eq:variance-inequality} is
\[
 \frac{\E(f'u)}{M^2}=26{R_a}+14{R_j}.
\]
For $R_2$, use $f=10\mathcal T-9w$, $\eta=7/2$. Expanding the variance
with \eqref{eq:covariance-T-w} gives
\[
 \frac{\Var(f)}{M^2}=100{V_T}-279{V_w}-180\chi+180{R_j}-90y.
\]
The first inequality in~\eqref{eq:profile-increment-comparisons} gives
$f'\ge15{\mathcal X}'+6w'=u(42a+12j)$, and hence
$\E(f'u)/M^2\ge42{R_a}+12{R_j}$. Replacing $\E(f'u)$ by this lower
bound preserves the inequality because $\eta>0$.
\end{proof}

\section{\texorpdfstring{Positivity of $J$}{Positivity of J}}\label{sec:positivity}

Until \zcref{sec:strict-inequality-completion}, unmarked expectations,
variances, and covariances are taken under $\rho$, as in
\zcref{sec:weighted-moments}.

Under the assumptions of \zcref{sec:weighted-moments}, it remains
to prove $O>0$ for every $0<\mu\le1$ for which the moment identity
can hold. Two variance estimates suffice when $\mu\le11/20$; the bounds
for $y$ exclude $\mu\ge29/40$. On the remaining interval, the proof
uses a quantitative covariance bound before combining the five inequalities.
The rational multipliers and polynomial checks are given in
\zcref{app:arithmetic}.

\subsection[The range 0 < mu <= 11/20]{The range $0<\mu\le11/20$}

Two applications of the variance inequality give the following uniform
positive lower bound.

\begin{proposition}\label{prop:small-positive}
Under the standing assumptions of \zcref{sec:weighted-moments}, if
$0<\mu\le11/20$, then $O>2/25$.
\end{proposition}

\begin{proof}
Put $\Xi=(\Sigma-\theta){\xi}/2+\theta/4$.
Since $\theta\le\kappa<11/40$,
we may use the upper bound $H_{11/40}$ in the variance inequality
\eqref{eq:variance-inequality}. Apply it to $(f,\eta)=(w-a,1/3)$ and
$(\mathcal T-w,3/10)$. The resulting quantities
\begin{equation}\label{eq:two-variance-inequalities}
\begin{aligned}
 L_a&={V_a}+{V_w}+2\chi-\frac23(4{R_a}+2{R_j})+\frac19H_{11/40},\\
 L_{\mathcal T}&={V_T}-3{V_w}-2\chi+2{R_j}-y-\frac35(4{R_a}+{R_j})
                                  +\frac9{100}H_{11/40}
\end{aligned}
\end{equation}
are nonnegative. The two variances, using \eqref{eq:covariance-T-w}
for the second, are $M^2({V_a}+{V_w}+2\chi)$ and
$M^2({V_T}-3{V_w}-2\chi+2{R_j}-y)$, while
\[
 \frac{\E[(w-a)'u]}{M^2}=4{R_a}+2{R_j},\qquad
 (\mathcal T-w)'\ge u(4a+j).
\]
The latter derivative bound follows from
\eqref{eq:profile-increment-comparisons}; replacing $\E(f'u)$ by this
lower bound preserves the inequality because $\eta>0$.

We combine $L_a$ and $L_{\mathcal T}$ with $O$ so that the remaining
variance and mixed-moment terms can be bounded using
$V_a\le\mu\chi$, $V_w\ge\chi/\mu$, and $R_j\ge kR_a$.
The multipliers below leave a positive coefficient of $\chi$ after
these comparisons throughout $0<\mu\le11/20$.
Substituting \eqref{eq:H-enlarged-bound} and ${R_a}=1/2-\chi$ gives
\begin{equation}\label{eq:small-mu-identity}
\begin{aligned}
 O-\frac{601}{100}L_a-\frac{11}{9}L_{\mathcal T}
 ={}&\frac{11}{9}(y^2+y)-\frac{659}{450}-\frac79\Xi\\
 &-\frac{301}{100}{V_a}+\frac{997}{1800}{V_w}
                    +\frac{409}{450}\chi+\frac2{25}{R_j}.
\end{aligned}
\end{equation}
By \eqref{eq:variance-comparisons} and $j\ge ka$,
\[
 {V_a}\le\frac{11}{20}\chi,\qquad
 {V_w}\ge\frac{20}{11}\chi,\qquad
 {R_j}\ge\frac9{11}{R_a}.
\]
Consequently,
\begin{equation}\label{eq:low-identity}
\begin{aligned}
 O&\ge\frac{11}{9}(y^2+y)-\frac{7087}{4950}-\frac79\Xi
          +\frac{38611}{198000}\chi\\
  &\ge\frac{11}{9}(y^2+y)-\frac{7087}{4950}-\frac79\Xi.
\end{aligned}
\end{equation}

Put $\bar {\xi}=k+3/2$ and $\gamma_*=69/20$. The bounds in this range give
\[
 k\ge\frac9{11},\qquad 1\le {\xi}\le \bar {\xi},\qquad
 a_0\ge\frac{29}{40},\qquad d+{\gamma}\ge \gamma_*,\qquad y\ge k>\frac{29}{40}.
\]
The lower bound \eqref{eq:lower-bound-y} for $y$ implies
\[
 \Xi\le {\xi}(y-a_0)+\theta\left(\frac{{\xi}^2}{d+{\gamma}}-\frac {\xi}2+\frac14\right)
 \le \bar {\xi}\left(y-\frac{29}{40}\right)
      +\kappa\left(\frac{\bar {\xi}^2}{\gamma_*}-\frac{\bar {\xi}}{2}+\frac14\right).
\]
To justify the $y$-dependent part of the second inequality, set
\[
 r=\frac{2g}{\mu}\in[0,1],\qquad
 D=\bar\xi\left(y-\frac{29}{40}\right)-\xi(y-a_0).
\]
Since
\[
 \frac{\partial D}{\partial y}=\bar\xi-\xi=\frac{1-r}{2}\ge0
 \qquad\text{and}\qquad y\ge k,
\]
it is enough to consider $D$ at $y=k$. As a function of $r$,
$D|_{y=k}$ is concave, and its endpoint values are
\[
 \left.D\right|_{y=k,r=0}
   =\frac{62-89\mu}{80\mu}>0,\qquad
 \left.D\right|_{y=k,r=1}
   =\frac{(\mu+2)(11-20\mu)}{80\mu}\ge0.
\]
Thus $D\ge0$ for $0<\mu\le11/20$. For the remaining term, the function
\[
 z\longmapsto\frac{z^2}{\gamma_*}-\frac z2+\frac14
\]
is nonnegative and increasing for $z\ge1$. Hence
$\theta\le\kappa$ and $\xi\le\bar\xi$ give the asserted bound for the
$\theta$-dependent term. 
Substitution in \eqref{eq:low-identity} gives
\[
 O\ge\frac{11}{9}(y^2+y)-\frac{7087}{4950}
 -\frac79\left[
 \bar {\xi}\left(y-\frac{29}{40}\right)
 +\kappa\left(\frac{\bar {\xi}^2}{\gamma_*}-\frac{\bar {\xi}}{2}+\frac14\right)\right].
\]
This expression increases for $y\ge k$, since its $y$-derivative is
at least $5k/3+1/18>0$. Set $y=k$ and expand to obtain
\begin{equation}\label{eq:small-mu-positive}
 O\ge\frac{26195840k^2+37004946k-42207909}{68310000}
 \ge\frac{4110277}{50094000}>\frac2{25}>0.
\end{equation}
The quadratic increases for $k\ge0$, so its minimum for $k\ge9/11$
is its value at $9/11$.
\end{proof}

\subsection[Exclusion of mu >= 29/40]{Exclusion of $\mu\ge29/40$}
\label{sec:large-mu-exclusion}

For every $0<\mu\le1$, the bound \eqref{eq:lower-bound-y} and
$d>\theta^2/\kappa$ imply
\begin{equation}\label{eq:theta-independent-bound}
 y>a_0+\frac{\theta^2}{2\kappa}-\frac{\theta {\xi}}{4-\mu}
 \ge a_0-\frac{\kappa {\xi}^2}{2(4-\mu)^2}.
\end{equation}
The second inequality is completion of the square in $\theta$.
We now compare \eqref{eq:theta-independent-bound} with the upper bound
for $y$ in \eqref{eq:global-mean-bounds} to exclude large $\mu$.

\begin{proposition}\label{prop:large-mu-exclusion}
Under the standing assumptions of \zcref{sec:weighted-moments}, one has
 $\mu < 29/40$.
\end{proposition}

\begin{proof}
Assume $29/40\le\mu\le1$. Writing $g=\mu/M-1$, and 
comparing \eqref{eq:theta-independent-bound} with the upper bound for
$y$, gives
\[
 y>\frac32-\frac\mu4-\frac{\mu}{2M}
 -\frac{\kappa}{2(4-\mu)^2M^2},\qquad
 y\le\frac{1-3\mu/2}{M}+\frac34.
\]
The difference of the lower expression and the upper expression is
\[
 \frac34-\frac\mu4-\frac{1-\mu}{M}
 -\frac{\kappa}{2(4-\mu)^2M^2}.
\]
It increases with $M$ for $0<\mu\le1$, as its derivative is
$(1-\mu)/M^2+\kappa/((4-\mu)^2M^3)>0$. At the smallest allowed
value $M=2\mu/(\mu+2)$ the difference is
\begin{equation}\label{eq:large-mu-exclusion}
 \Delta(\mu)=\frac{\mu+5}{4}-\frac1\mu
 -\frac{169}{5000}
 \left[\frac{\mu+2}{\mu(4-\mu)}\right]^2.
\end{equation}
The bracket is positive and decreases on $(0,1]$, because its
derivative has the sign of $\mu^2+4\mu-8<0$. Thus $\Delta$ increases.
In particular,
\[
 \Delta(29/40)=\frac{85671561}{11545920800}>0.
\]
This contradicts the simultaneous lower and upper bounds for $y$
when $\mu\ge29/40$.
\end{proof}

The remaining argument concerns the intermediate parameter domain
\begin{equation}\label{eq:middle-domain}
 \mathcal D=\left\{\frac{11}{20}\le\mu\le\frac{29}{40},\quad
                         0\le g\le\frac\mu2\right\}.
\end{equation}
In this range set
\begin{equation}\label{eq:middle-basic-scalars}
 \xi=\frac{1+g}{\mu},\qquad \gamma=4-\mu,\qquad
 a_0=1-\mu/4-g/2,\qquad \kappa=169/625.
\end{equation}
\subsection[Eliminating Sigma and theta]{Eliminating $\Sigma$ and $\theta$}

The dependence on $s$ enters only through $\Sigma$. The next lemma uses
\eqref{eq:lower-bound-y} to eliminate $\Sigma$ and then completes the square
in $\theta$.

\begin{lemma}\label{lem:middle-quadratic-reduction}
Let $(\mu,g)\in\mathcal D$, and let $\xi,\gamma,a_0,\kappa$ be as in
\eqref{eq:middle-basic-scalars}. Suppose that $\theta\ge0$, $\Sigma\ge d\ge\theta^2/\kappa$, and
\[
 y\ge a_0+\frac\Sigma2-\frac{\theta {\xi}}{d+{\gamma}}.
\]
Suppose also that $\mathsf A,\mathsf B,\mathsf C,\Lambda$ depend only on
$(\mu,g)$, that $\mathsf A>0$ and $\Lambda>0$, and that
\begin{equation}\label{eq:abstract-middle-bound}
 O\ge\mathsf A y^2+\mathsf B y+\mathsf C
      -\Lambda\left(\frac{(\Sigma-\theta){\xi}}{2}+\frac\theta4\right).
\end{equation}
Let ${\ell}>\Lambda {\xi}$, and define
\begin{equation}\label{eq:middle-quadratic-coefficients}
\begin{gathered}
 \mathfrak a=\frac{{\ell}-\Lambda {\xi}}{2\kappa},\qquad
 \mathfrak b=\Lambda\left(\frac {\xi}2-\frac14\right)-\frac{{\ell}{\xi}}{{\gamma}},\\
 \mathfrak c={\ell}a_0+\mathsf C-\frac{({\ell}-\mathsf B)^2}{4\mathsf A}.
\end{gathered}
\end{equation}
One has $\mathfrak a>0$ and
\begin{equation}\label{eq:middle-discriminant-bound}
 O\ge\mathfrak a\theta^2+\mathfrak b\theta+\mathfrak c
   \ge\frac{4\mathfrak a\mathfrak c-\mathfrak b^2}{4\mathfrak a}.
\end{equation}
\end{lemma}

\begin{proof}
We use a supporting line of slope $\ell>\Lambda\xi$ so that
substituting the lower bound for $y$ leaves a positive coefficient
of $\Sigma$. The supporting-line bound
\[
 \mathsf Ay^2+\mathsf By
 \ge {\ell}y-\frac{({\ell}-\mathsf B)^2}{4\mathsf A}
\]
follows by completing the square. Substitution of the assumed lower bound for $y$ gives
\[
\begin{aligned}
 O\ge{}&{\ell}a_0+\mathsf C-\frac{({\ell}-\mathsf B)^2}{4\mathsf A}
       +\frac{{\ell}-\Lambda {\xi}}{2}\Sigma\\
 &+\Lambda\theta\left(\frac {\xi}2-\frac14\right)
       -\frac{{\ell}\theta {\xi}}{d+{\gamma}}.
\end{aligned}
\]
Use $\Sigma\ge\theta^2/\kappa$ and $(d+{\gamma})^{-1}\le {\gamma}^{-1}$.
The signs ${\ell}-\Lambda {\xi}>0$ and ${\ell}\theta {\xi}\ge0$ give the first bound in
\eqref{eq:middle-discriminant-bound}. Completing the square gives the second.
\end{proof}

\subsection[A lower bound for the covariance]{A lower bound for $\chi$}

The final linear combination retains a positive multiple of $\chi$, so we
need a quantitative lower bound for this covariance. Since $V_a\le2g\chi$,
a small covariance forces a small variance $V_a$. We first bound
$\mathcal H$ above, then apply integration by parts with test function $a$
to obtain a lower bound for $\chi$.

\begin{lemma}\label{lem:middle-covariance-floor}
Under the standing assumptions of \zcref{sec:weighted-moments}, for
$(\mu,g)\in\mathcal D$, put
$\overline H(\mu)=(9+8\mu)/(4\mu^2)$. Then
\begin{equation}\label{eq:middle-covariance-floor}
 \mathcal H<\overline H(\mu)-{R_a},\qquad
 \chi>q_0(\mu,g)=\frac{3+8\mu-8g}{50}.
\end{equation}
\end{lemma}

\begin{proof}
Write ${R_a}=1/2-\chi$. By \eqref{eq:H-first-upper-bound},
$0<u'\le1$, and the moment identity $\E_\rho w=M/2$,
\begin{align*}
 \mathcal H
 &\le\frac{\E_\rho[u'^2+w(4u'-a+s(v)-
                       \theta\sech^2(\alpha v))]}{M^2}\\
 &\le\frac{\E_\rho u'+4\E_\rho w-
                    \E_\rho(aw)+\E_\rho[s(v)w]}{M^2}
 ={\xi}\left(3+y+\frac\Sigma2\right)-{R_a}.
\end{align*}
Let $\overline y=3/4+(1-3\mu/2){\xi}$ be the upper bound for $y$ in
\eqref{eq:global-mean-bounds}. The lower bound
\eqref{eq:lower-bound-y} gives
\[
 y\le\overline y,\qquad
 \frac\Sigma2\le\overline y-a_0+\frac{\theta {\xi}}{d+{\gamma}}
              \le\overline y-a_0+\frac{\kappa {\xi}}{{\gamma}}.
\]
Hence $\mathcal H\le H_0-{R_a}$, where
\begin{equation}\label{eq:covariance-H0}
 H_0={\xi}\left(3+2\overline y-a_0+\frac{\kappa {\xi}}{{\gamma}}\right)
    ={\xi}\left(\frac12+\frac\mu4-\frac{5g}{2}
                    +\left(2+\frac\kappa {\gamma}\right){\xi}\right).
\end{equation}
Writing $\varepsilon_0=\kappa/{\gamma}$, expansion gives
\[
\begin{aligned}
 \mu^2H_0={}&2+\varepsilon_0+\frac\mu2+\frac{\mu^2}{4}+\left(4+2\varepsilon_0-2\mu+\frac{\mu^2}{4}\right)g
  +\left(2+\varepsilon_0-\frac{5\mu}{2}\right)g^2.
\end{aligned}
\]
Both nonconstant coefficients are positive, so $H_0$ increases in $g$.
At $g=\mu/2$, put $\delta=\mu-11/20\in[0,7/40]$. Direct expansion yields
\[
\begin{aligned}
 2500\mu^2{\gamma}\bigl(\overline H-H_0(\mu,\mu/2)\bigr)
 ={}&\frac{176673}{3200}+\frac{62123}{80}\delta+8531\delta^2+\delta^3(1625-1250\delta)>0.
\end{aligned}
\]
Thus $\mathcal H<\overline H-{R_a}$.

For the test function $a$, the identity $a'=-2au$ and
\eqref{eq:variance-inequality} give
\[
 \frac{\E_\rho(a'u)}{M^2}=-2R_a,\qquad
 V_a+4\eta R_a+\eta^2\mathcal H\ge0\quad(\eta\in\R).
\]
Nonnegativity of this polynomial for every real $\eta$ implies
$V_a\mathcal H\ge4R_a^2$, whereas \eqref{eq:convex-covariance} gives
$V_a\le2g\chi$. Since the
density of $\rho$ is positive and $au^2>0$ on $v>0$, one has
${R_a}=\E(au^2)/M^2>0$, hence $g\chi>0$.
We next turn the resulting quadratic constraint on $\chi$ into a
linear lower bound. Expanding about $\chi=1/10$ will leave a
nonnegative square that can be discarded. Set
\[
 h_0=\overline H-\frac12,\qquad
 \mathcal B=\frac{16}{5}+\left(2h_0+\frac25\right)g.
\]
Using $\mathcal H<h_0+\chi$ gives the expansion
\begin{equation}\label{eq:covariance-square}
\begin{aligned}
 0&>4(1/2-\chi)^2-2g\chi(h_0+\chi)\\
  &=\frac{24}{25}-\mathcal B\chi
      +(4-2g)\left(\chi-\frac1{10}\right)^2+\frac g{50}.
\end{aligned}
\end{equation}
Since $g\le29/80<2$, the square and $g/50$ are nonnegative, so
$\mathcal B\chi>24/25$. As $\chi>0$, also $\mathcal B>0$, and hence
$\chi>24/(25\mathcal B)$.

It remains to show $\mathcal Bq_0<24/25$. The function $\overline H$
decreases, and $\overline H(29/40)=5920/841>7$, so $h_0>13/2$.
The derivative of $\mathcal Bq_0$ with respect to $g$ decreases in $g$,
and at $g=\mu/2$ it is $(15h_0-61)/125>0$. Thus the maximum is at that
edge, where
\[
 \frac{24}{25}-\left.\mathcal Bq_0\right|_{g=\mu/2}
 =\frac{P(\mu)}{1000\mu},\qquad
 P(\mu)=24\mu^3-398\mu^2+468\mu-135.
\]
Now $P''(\mu)=144\mu-796<0$, and
\[
 P(11/20)=\frac{2999}{500}>4,\qquad
 P(29/40)=\frac{33977}{8000}>4.
\]
Concavity gives $P>4$ throughout the interval for $\mu$. Consequently
$\chi>24/(25\mathcal B)>q_0$, as required.
\end{proof}

\subsection{Combining the five moment inequalities}

We choose the multipliers so that $V_a,V_T,V_w,R_j$ cancel and $\chi$
has a positive coefficient. This allows us to use the lower bound from
\zcref{lem:middle-covariance-floor}. The next lemma gives the resulting
algebraic identity and verifies the signs and discriminant estimate used
below. The explicit rational formulas  are given in
\zcref{app:arithmetic}.
Here $\Xi=(\Sigma-\theta)\xi/2+\theta/4$.

\begin{lemma}\label{lem:algebraic-combination}
For $(\mu,g)\in\mathcal D$, let
$\lambda_V,\lambda_C,\lambda_1,\lambda_2,\lambda_{1/3}$ be the rational
functions in \eqref{eq:middle-weights}, let
$\mathsf A,\mathsf B,\mathsf C^0,\Lambda$ be given by
\eqref{eq:middle-raw-coefficients}, and let $\mathcal R$ be given by
\eqref{eq:middle-covariance-coefficient}. The five multipliers are positive,
and
\begin{align}\label{eq:middle-moment-identity}
 O={}&\mathsf Ay^2+\mathsf By+\mathsf C^0
 -\Lambda\left(\frac{(\Sigma-\theta){\xi}}{2}+\frac\theta4\right)
 +\mathcal R\chi\notag\\
 &+\lambda_VR_\square+\lambda_CR_C+\lambda_1R_1
                     +\lambda_2R_2+\lambda_{1/3}R_{1/3}.
\end{align}
After substituting \eqref{eq:normalized-J},
\eqref{eq:combined-square-expansion}, \eqref{eq:nonnegative-moments},
\eqref{eq:H-enlarged-bound}, and $R_a=1/2-\chi$,
\eqref{eq:middle-moment-identity} is an algebraic identity. The scalar
functions depend only on $\mu,g$ and satisfy 
\begin{equation}\label{eq:algebraic-interface-signs}
 \mathsf A>9/10,\qquad 0<\Lambda<11/10,\qquad \mathcal R>0.
\end{equation}
Define
\begin{equation}\label{eq:middle-adjusted-constant}
 \mathsf C_\dagger=\mathsf C^0+
                       \frac{3+8\mu-8g}{50}\mathcal R.
\end{equation}
With $\ell=10/3$ and $\mathsf C=\mathsf C_\dagger$, the coefficients in
\eqref{eq:middle-quadratic-coefficients} satisfy
$4\mathfrak a\mathfrak c-\mathfrak b^2>0$ throughout $\mathcal D$.
\end{lemma}

The covariance estimate and the algebraic identity now allow us to apply
\zcref{lem:middle-quadratic-reduction}.

\begin{proposition}\label{prop:middle-positive}
Under the standing assumptions of \zcref{sec:weighted-moments}, if
$11/20\le\mu\le29/40$, then $O>0$.
\end{proposition}

\begin{proof}
In \eqref{eq:middle-moment-identity}, discard the five nonnegative terms
using \zcref{lem:five-nonnegative-moments}, and use
$\chi>(3+8\mu-8g)/50$ from
\zcref{lem:middle-covariance-floor}. Since $\mathcal R>0$, the result
is \eqref{eq:abstract-middle-bound} with
$\mathsf C=\mathsf C_\dagger$. Apply
\zcref{lem:middle-quadratic-reduction} with $\ell=10/3$ and the
coefficients $\mathsf A,\mathsf B,\Lambda$ from
\zcref{lem:algebraic-combination}. The signs
$\mathsf A>0$ and $\Lambda>0$ follow from
\eqref{eq:algebraic-interface-signs}.
\zcref{lem:general-mean-bound} supplies the required lower bound
for $y$ and $\Sigma\ge d>\theta^2/\kappa$, while
\eqref{eq:moment-parameter-bounds} gives $\theta\ge0$.
Moreover, $\xi\le1/\mu+1/2\le51/22$ on $\mathcal D$, so
\[
 \ell-\Lambda\xi>\frac{10}{3}-\frac{51}{20}=\frac{47}{60}>0.
\]
All hypotheses of the reduction are therefore satisfied. Its lower
bound \eqref{eq:middle-discriminant-bound} is strictly positive by
the discriminant estimate in \zcref{lem:algebraic-combination}.
Hence $O>0$.
\end{proof}

\subsection{Completion of the strict inequality}\label{sec:strict-inequality-completion}

\begin{proof}[Proof of \zcref{thm:derivative-moments}]
In this proof, $\E$ again denotes expectation under the original diffusion
law; weighted expectations are written $\E_\rho$.
Fix $t$ and keep the original coefficient in the PDE and diffusion.
Construct the normalized functions using the fixed value $m_0$ from the
hypotheses of \zcref{thm:derivative-moments}. By Girsanov's theorem and
boundedness of the drift, the law of $Y_t$ is equivalent to the law of
$h+W_t$. It therefore gives positive mass to every nonempty open interval.
The function
$C_x(t,\cdot)$ is not identically zero: $C>0$, whereas the bounded
increasing function $B$ has finite limits at both spatial infinities.
Thus $F(t)>0$. In particular, $m_0=0$ is incompatible with the
moment identity $F-2m_0G=0$.

For $m_0>0$, set $m_{\mathrm{rev}}(r)=m(L-r)$.
The hypothesis on $(t,L]$ gives $m_{\mathrm{rev}}\ge m_0$
almost everywhere on $[0,L-t]$. Thus
\zcref{lem:normalized-estimates} applies at reverse time $L-t$
with this fixed value $m_0$, giving the scale and normalized variables:
\[
 \begin{gathered}
 c_0=C(t,0),\qquad
 \lambda^2=-\frac{C_{xx}(t,0)}{2c_0}>0,\qquad
 \mu=\frac{m_0c_0}{\lambda^2},\\
 v=\lambda x,\qquad
 a(v)=\frac{m_0C(t,v/\lambda)}{\lambda^2},\qquad
 u(v)=-\frac{C_x(t,v/\lambda)}{2\lambda C(t,v/\lambda)}.
 \end{gathered}
\]
The PDE estimates give $0<\mu\le1$ and the differential inequalities
\eqref{eq:moment-profile-inequalities}. The probability measure $\rho$ is defined, for every
bounded Borel function $f$ on $[0,\infty)$, by
\[
 \E_\rho f(v)=\frac{\E[C(t,Y_t)^2f(\lambda|Y_t|)]}{S(t)}.
\]
Apply \zcref{lem:lower-density} with $m_*=m_0$ and this same
coefficient $m$. Its bounds before and after $t$ are exactly the
two one-sided hypotheses of \zcref{thm:derivative-moments}.
It gives \eqref{eq:moment-density}, with $d=1/(\lambda^2t)$,
$\alpha=h/(\lambda t)$, and \eqref{eq:coefficient-s-bounds}.
Finally, \zcref{prop:parameter-bounds} applies to the original
diffusion and the hypothesis $N(t)\le0$, giving
\eqref{eq:moment-parameter-bounds}.

By \eqref{eq:weighted-moment-scaling}, the hypothesis $F-2m_0G=0$ is
exactly the moment identity \eqref{eq:moment-constraint}.
By \zcref{prop:large-mu-exclusion}, one has $\mu<29/40$.
The ranges $0<\mu\le11/20$ and $11/20<\mu<29/40$ are covered by
\zcref{prop:small-positive,prop:middle-positive}, respectively.
Hence $O=\E_\rho {\mathcal J}/M^2>0$ whenever the moment identity holds.
Since $\lambda,S,M$ are positive, \eqref{eq:weighted-moment-scaling}
gives $J(t;m_0)>0$.
\end{proof}

\section{Proof of the main theorem}\label{sec:structure}

The proof combines the gap argument of \cite{Zhou2026FRSB} and
\cite[Propositions~9.1--9.2]{Lopatto2026} with the positive-field inequality
in \zcref{thm:derivative-moments}. After excluding gaps, we retain
one-sided values in the regularity argument and apply the same inequality
to the density and endpoint masses.

\begin{proof}[Proof of \zcref{thm:main}]
\emph{Step 1: Connected support.}
Write $\nu=\nu_{\beta,h}$. Suppose that its support is disconnected.
Since it is compact and avoids $0,1$, there exist support points $0<a<b<1$ such that
$(a,b)$ is disjoint from the support. Put ${t_-}=La$, ${t_+}=Lb$.
Let $m_0$ be the constant value of $m$ on $({t_-},{t_+})$.
For every $t$ in this gap, monotonicity gives $m\le m_0$ before
$t$ and $m\ge m_0$ after $t$. Thus the coefficient hypotheses of
\zcref{thm:derivative-moments} hold throughout the gap.
\zcref{lem:variational-support} gives
\[
 Q({t_-})/{t_-}=Q({t_+})/{t_+}=L^{-1},\qquad N({t_-}),N({t_+})\ge0.
\]
Moreover,
\[
 \int_{t_-}^{t_+}\frac{N(t)}{t^2}\dd t
 =-\left[\frac{Q(t)}t\right]_{t_-}^{t_+}=0.
\]
If $N$ is not identically zero, these facts force a negative interior
minimum at some $t_0$. At that point, the identities
\[
 N'=-tQ'',\qquad N''=-Q''-tJ(t;m_0)
\]
give $F(t_0)-2m_0G(t_0)=Q''(t_0)=0$ and $J(t_0;m_0)\le0$.
Together with $N(t_0)<0$, this contradicts
\eqref{eq:strict-derivative-moment-inequality}.
If $N\equiv0$, then $F-2m_0G=Q''=0$ and $J(t;m_0)=0$
throughout the gap, giving the same contradiction.
Thus the support is connected. A nonempty compact connected subset
of $(0,1)$ is a closed interval, possibly a singleton; write it as
$[q_-,q_+]$. A probability measure has singleton support exactly when it is a point
mass. This proves the assertion about replica symmetry. For the remaining
conclusions, assume $q_-<q_+$.

\medskip
\emph{Step 2: A smooth positive density.}
We follow the proof of \cite[Theorem~2(ii)]{AuffingerChen2015Properties}:
first identify $m$ with $F/(2G)$ and then iterate the differentiated moment
identities, retaining the one-sided values at both endpoints.
Put ${t_-}=Lq_-$ and ${t_+}=Lq_+$.
By the support conditions, $Q(t)=t/L$ on $[{t_-},{t_+}]$.
By \eqref{eq:ito-Q}--\eqref{eq:ito-S},
\begin{equation}\label{eq:cdf-quotient}
 m(t)=\frac{F(t)}{2G(t)}\qquad\text{for almost every }t\in({t_-},{t_+}).
\end{equation}
Both $F$ and $G$ are Lipschitz. Since $C>0$, the continuous function
$G$ has a positive minimum on $[{t_-},{t_+}]$. The right side of
\eqref{eq:cdf-quotient} is consequently Lipschitz on the closed interval.
Right-continuity of the cumulative distribution extends the equality
to every $t\in[{t_-},{t_+})$; its value at ${t_+}$ is $m({t_+}-)$.
Thus there is no atomic or singular continuous part in $(q_-,q_+)$.

For the endpoint bootstrap, define the continuous representative
\[
 \bar m(t)=\frac{F(t)}{2G(t)},\qquad t\in[t_-,t_+].
\]
Then $\bar m(t)=m(t)$ for $t\in[t_-,t_+)$ and
$\bar m(t_+)=m(t_+-)$, so $m=\bar m$ almost everywhere on
$[t_-,t_+]$. Hence the polynomial-moment identities in
\zcref{lem:ito-moments} may be written, for $t_-\le s\le t\le t_+$, as
\[
 M_P(t)-M_P(s)
 =\int_s^t\bigl(M_{P_0}(r)+\bar m(r)M_{P_1}(r)\bigr)\dd r.
\]
In particular, the right sides of
\eqref{eq:ito-F}--\eqref{eq:ito-G} extend continuously to the endpoints
with $m$ there interpreted as $\bar m$. We use this convention in the
remainder of this step. Differentiating 
\eqref{eq:cdf-quotient} gives
\begin{equation}\label{eq:cdf-derivative}
 \begin{aligned}
  2G(t)m'(t)&=F'(t)-2m(t)G'(t)\\
  &=\E[C_{xx}^2-12mCC_x^2+6m^2C^4]=J(t;m(t)).
 \end{aligned}
\end{equation}
For smoothness, consider simultaneously all moment functions
\[
 t\longmapsto
 \E\!\left[P\bigl(B(t,Y_t),C(t,Y_t),C_x(t,Y_t),\ldots\bigr)\right],
\]
where $P$ is a polynomial in finitely many positive-order spatial
derivatives of $U$. 
The quotient above first gives Lipschitz continuity of the interior
extension of $m$. We bootstrap this regularity using the polynomial-moment
integral identities established from \eqref{eq:all-derivative-ito} in
\zcref{lem:ito-moments}.

Apply these identities first on compact subintervals of
$({t_-},{t_+})$. If $m$ and all the polynomial moment functions are $C^k$,
the drift of each such moment is $C^k$, so every moment is $C^{k+1}$; the
quotient $m=F/(2G)$ then makes $m$ $C^{k+1}$. Starting with continuity proves
smoothness by induction. At each stage the integrands extend continuously to
the endpoints, so the same integral identities give the corresponding
one-sided derivatives there.

On $[{t_-},{t_+}]$ one has $Q=t/L$ and $S=1/L$, so $N=0$.
The quotient identity gives $F=2m_0G$ also at both endpoints, with
$m_0$ equal to the respective one-sided value of $m$ from the interior.
Each one-sided value satisfies the hypotheses on $m_0$ in
\zcref{thm:derivative-moments}.
\zcref{thm:derivative-moments} therefore gives
$J(t;\bar m(t))>0$ at every point of the closed interval. Equation~\eqref{eq:cdf-derivative} shows
that $m'>0$ there. In the original coordinate the density is
$f(q)=L m'(Lq)$.

\medskip
\emph{Step 3: Endpoint masses.}
By Girsanov's theorem, the diffusion law at ${t_-}>0$ gives positive mass
to every nonempty open interval. Since $C_x({t_-},\cdot)$ is not identically
zero, $F({t_-})>0$.
Thus \eqref{eq:cdf-quotient} gives
$\nu(\{q_-\})=m({t_-})>0$.
For the upper endpoint, we use the moment calculation for the terminal
profile $\log\cosh$ from the proof of 
\cite[Theorem~4]{AuffingerChen2015Properties}, in the form used in
\cite[Proposition~9.3]{Lopatto2026} and
\cite[Proposition~2.1]{Chen2026EndpointMass}. The quantitative input is 
\eqref{eq:theta-independent-bound}, obtained from
\zcref{lem:general-mean-bound}.

Put $m_+=m({t_+}-)\in(0,1]$. For $t>{t_+}$, the cumulative distribution function equals one.
The Cole--Hopf formula gives
$U({t_+},x)=\log\cosh x+(L-{t_+})/2$. Thus $c_0=\lambda=1$, $\mu=m_+$, and
\[
 u(v)=\tanh v,\qquad a(v)=m_+\sech^2v,\qquad
 j(v)=(1-m_+)\sech^2v.
\]
The moment identity $\E_\rho a=2\E_\rho u^2$ at ${t_+}$, evaluated with
$m_0=m({t_+}-)$, gives
\begin{equation}\label{eq:terminal-moments}
 M=\frac{2m_+}{m_++2},\qquad g=\frac{m_+}{2},\qquad
 \xi=\frac1M=\frac{m_++2}{2m_+},\qquad
 y=\frac{1-m_+}{m_+}.
\end{equation}
Since $N(t_+)=0$, \zcref{prop:parameter-bounds} gives the parameter
bounds \eqref{eq:moment-parameter-bounds}. The terminal moment identity
gives \eqref{eq:moment-constraint}, so \eqref{eq:theta-independent-bound}
applies with $m_0=m(t_+-)$, including when $m_+=1$.
Using \eqref{eq:terminal-moments} and $a_0=1-m_+/2$ in
\eqref{eq:theta-independent-bound} gives
\[
 \frac{1-m_+}{m_+}>1-\frac{m_+}{2}
 -\frac{\kappa}{2(4-m_+)^2}
       \left(\frac{m_++2}{2m_+}\right)^2.
\]
Move the left side to the right and use $\kappa=169/625$ to obtain
\begin{equation}\label{eq:terminal-H}
 H(m_+)=2-\frac{m_+}{2}-\frac1{m_+}
 -\frac{169}{5000}
  \left(\frac{m_++2}{m_+(4-m_+)}\right)^2<0.
\end{equation}
For $0<m\le1$,
\[
 \left(2-\frac m2-\frac1m\right)'=\frac1{m^2}-\frac12>0,
 \qquad
 \left(\frac{m+2}{m(4-m)}\right)'
 =\frac{m^2+4m-8}{m^2(4-m)^2}<0.
\]
If $m_+\ge5/8$, these monotonicities imply
\[
 H(m_+)>\frac7{80}-\frac{169}{5000}\left(\frac54\right)^2
 =\frac{111}{3200}>0,
\]
since $(m+2)/(m(4-m))\le56/45<5/4$ at $m\ge5/8$.
This contradicts \eqref{eq:terminal-H}. Hence $m_+<5/8$ and
$\nu(\{q_+\})=1-m_+>3/8$.
Together with Step~2, these endpoint masses give the asserted
measure decomposition and complete the proof.
\end{proof}

\appendix
\zcsetup{countertype={section=appendix}}

\section{The algebraic estimates}\label{app:arithmetic}

The first part of this appendix proves the algebraic lemma used for
intermediate $\mu$. The second verifies the three polynomial inequalities:
two for the Gaussian comparison and one for the discriminant estimate
in the intermediate-$\mu$ moment argument.  All coefficients are specified by rational formulas and finite
sums.

\subsection{Verification of the linear combination}

\begin{proof}[Proof of \zcref{lem:algebraic-combination}]
The coefficients are chosen to cancel $V_a,V_T,V_w,R_j$. This leaves only
$y$, the scalar term $\Xi$, and the covariance $\chi$, which are controlled
by \eqref{eq:lower-bound-y} and \zcref{lem:middle-covariance-floor}. Substitution of the five nonnegative
quantities shows that the cancellation equations are
\begin{equation}\label{eq:multiplier-cancellations}
\begin{aligned}
 (25-20\mu)\lambda_V+\Lambda+100\lambda_2&=2
       &&[V_T],\\
 (16-20/\mu)\lambda_V-\lambda_C+36\lambda_1+\lambda_{1/3}&=3
       &&[V_a],\\
 (80\mu-99)\lambda_V+49\lambda_1-279\lambda_2
       +\lambda_{1/3}/9+(\kappa-4)\Lambda&=0
       &&[V_w],\\
 (90-80\mu)\lambda_V-(364/5)\lambda_1+96\lambda_2+8\Lambda&=0
       &&[R_j].
\end{aligned}
\end{equation}
where $\Lambda=(169/25)\lambda_1+(49/4)\lambda_2$.
After substituting this definition of $\Lambda$, the four matching
equations leave one free multiplier. We use this freedom to obtain
positive multipliers and the strict bound
$4\mathfrak a\mathfrak c-\mathfrak b^2>0$ required in
\zcref{lem:algebraic-combination}. Choose
\[
 \lambda_2=\frac{24-31\mu}{2250}
\]
and solve for $\lambda_V,\lambda_C,\lambda_1,\lambda_{1/3}$.
The system is nonsingular throughout the stated interval because
$D_0=207-176\mu>0$. Solving gives
\eqref{eq:middle-weights}; the estimates below verify the required signs and discriminant bound.

Define
\begin{align*}
 P_3(\mu)&=218058730848\mu^3-299524062427\mu^2
                   +115931336840\mu-291720000,\\
 P_2(\mu)&=28181247872\mu^2-36876645103\mu+12396836760.
\end{align*}
For $(R_\square,R_C,R_1,R_2,R_{1/3})$, take the positive multipliers
\begin{align}\label{eq:middle-weights}
 \lambda_V&=\frac{203453\mu+4488}{22500D_0},&
 \lambda_C&=\frac{P_3(\mu)}{73125000\mu D_0},\notag\\
 \lambda_1&=\frac{111576-29909\mu-63240\mu^2}{4680D_0},&
 \lambda_2&=\frac{24-31\mu}{2250},\notag\\
 \lambda_{1/3}&=\frac{P_2(\mu)}{8125000D_0}.
\end{align}
All these multipliers are positive for $11/20\le\mu\le29/40$.
Indeed, $D_0>0$, $24-31\mu>0$, and the numerator of $\lambda_1$ decreases to the
positive value $1133029/20$. The numerator and denominator of $\lambda_V$ are positive. The leading coefficient of $P_2$ is positive and its
discriminant is
\[
 -37546364196511618271<0.
\]
For $P_3$, put $z_0=\mu-11/20\in[0,7/40]$. Its exact expansion is
\[
 P_3(11/20+z_0)
 =\frac{18288015445337}{2000}
 -\frac{782841679257}{50}z_0
 +\frac{301364217361}{5}z_0^2+218058730848z_0^3.
\]
In particular, $P_3>10^9(9-16z_0)>0$.

Define
\begin{align}\label{eq:middle-raw-coefficients}
 \Lambda&=\frac{169}{25}\lambda_1+\frac{49}{4}\lambda_2,&
 \mathsf A&=2+20\mu\lambda_V-\Lambda,\notag\\
 \mathsf B&=5\lambda_V+90\lambda_2,&
 \mathsf C^0&=-9+20\mu\lambda_V(k^2-Q_*)
                  +\frac{507}{10}\lambda_1+\frac{931}{8}\lambda_2.
\end{align}
The upper bound $Q_*$ for the normalized second moment of $j-ka$ is
given explicitly in \eqref{eq:quadratic-bound-j-ka}. 
Let
\begin{align}\label{eq:middle-covariance-coefficient}
 P_R(\mu)&=-218058730848\mu^3+580543404659\mu^2\notag\\
        &\qquad-492344180708\mu+135174628560,\notag\\
 \mathcal R(\mu,g)&=\frac{P_R(\mu)}{73125000D_0}
                      +2\lambda_C\left(\frac\mu2-g\right).
\end{align}

Substituting these formulas into \eqref{eq:multiplier-cancellations}
verifies the four cancellations. The remaining $y^2$, $y$, and constant
coefficients give \eqref{eq:middle-raw-coefficients}; the coefficient of
$\chi$ is $\mathcal R$. For example, the coefficient of $V_T$ is
$(25-20\mu)\lambda_V+(169/25)\lambda_1+(449/4)\lambda_2=2$,
and the term in $\Xi$ is $\Lambda\Xi$. Moving that term to the other
side gives its minus sign in \eqref{eq:middle-moment-identity}.
Thus the displayed formulas give the exact identity in the statement.

The coefficient $\mathcal R$ of $\chi$ is positive. To see this, write
$z_0=\mu-7/10\in[-3/20,1/40]$. Then
\[
 P_R(7/10+z_0)=\frac{102912833223}{500}
 -\frac{3243713299}{25}z_0
 +\frac{613100349391}{5}z_0^2-218058730848z_0^3.
\]
The constant exceeds $200\cdot10^6$. The linear term is bounded below
by $-130\cdot10^6/40$, and the last two terms together are
nonnegative, since their quadratic factor is at least
$10^9(122-219/40)>0$. Thus $P_R>0$, and $\mathcal R>0$ follows from
$g\le\mu/2$ and $\lambda_C>0$.

Set $\ell=10/3$ as in the statement. The formulas above give
\[
 \Lambda=\frac{846960-455113\mu-277388\mu^2}{4500D_0},\qquad
 \frac{11}{10}-\Lambda
 =\frac{277388\mu^2-416087\mu+177690}{4500D_0}>0.
\]
For the last sign, put $\delta=\mu-3/4<0$; the numerator is
$277388\delta^2-5\delta+43311/2>0$. Positivity of the weights now gives
\begin{equation}\label{eq:fixed-slope-signs}
 0<\Lambda<\frac{11}{10},\qquad
 \mathsf A=2+20\mu\lambda_V-\Lambda>\frac9{10},\qquad
 {\ell}-\Lambda {\xi}>\frac{10}{3}-\frac{51}{20}=\frac{47}{60}>0,
\end{equation}
since ${\xi}\le1/\mu+1/2\le51/22$ on $\mathcal D$.

Apply \eqref{eq:middle-quadratic-coefficients} with ${\ell}=10/3$ and the
constant \eqref{eq:middle-adjusted-constant}, obtaining
$\mathfrak a_{\mathrm f},\mathfrak b_{\mathrm f},\mathfrak c_{\mathrm f}$.
Put
\[
 P_A(\mu)=218240\mu^2-222187\mu+203208,
 \qquad \mathsf A=\frac{P_A(\mu)}{900D_0},
\]
so $P_A>0$ by \eqref{eq:fixed-slope-signs}. Define $N_{\mathrm f}$ by
\begin{equation}\label{eq:middle-discriminant}
 \Delta_{\mathrm f}=4\mathfrak a_{\mathrm f}\mathfrak c_{\mathrm f}
                         -\mathfrak b_{\mathrm f}^{2}
 =\frac{N_{\mathrm f}(\mu,g)}{
 17795700000000\mu^2(\mu-4)^2(176\mu-207)^2P_A(\mu)}.
\end{equation}
These definitions give an integer polynomial $N_{\mathrm f}$ of bidegree
$(10,3)$. The denominator in \eqref{eq:middle-discriminant} is
positive on $\mathcal D$, and \zcref{lem:arithmetic-certificate}
proves $N_{\mathrm f}>0$ there.
This proves the final assertion of \zcref{lem:algebraic-combination}.

\end{proof}

\subsection{Polynomial positivity}

We recall the classical Bernstein coefficient bound and change-of-basis
formula; see \cite{FaroukiRajan1988}. Allowing a separate degree in the
second variable for each first-variable coefficient is simply an
iterated use of the univariate formula. We include the short proof.

\begin{lemma}\label{lem:bernstein-coefficients}
Let $n,m$ be nonnegative integers. For $0\le i\le n$, put
\[
 B_{i,n}(z)=\binom ni z^i(1-z)^{n-i}.
\]
Let $p(z_1,z_2)=\sum_{k,l}a_{kl}z_1^kz_2^l$ be a polynomial with
real coefficients and bidegree at most $(n,m)$, and choose integers
$n_i\ge m$ for $0\le i\le n$. Then
\begin{equation}\label{eq:bernstein-coefficients}
 \begin{gathered}
 p(z_1,z_2)=\sum_{i=0}^n B_{i,n}(z_1)
                       \sum_{j=0}^{n_i}\beta_{ij}B_{j,n_i}(z_2),\qquad \beta_{ij}=\sum_{\substack{k\le i\\l\le j}}
 a_{kl}\frac{\binom ik}{\binom nk}\frac{\binom jl}{\binom{n_i}l}.
 \end{gathered}
\end{equation}
In particular $p(z_1,z_2)\ge\min_{i,j}\beta_{ij}$ on $[0,1]^2$.
The univariate version omits the second variable and its index.
\end{lemma}

\begin{proof}
The binomial theorem gives
\[
 z^k=\sum_{i=k}^n\frac{\binom ik}{\binom nk}B_{i,n}(z).
\]
First apply this identity in $z_1$. Each resulting coefficient is a
polynomial in $z_2$ of degree at most $m$; expand the $i$th such polynomial
in the Bernstein basis of degree $n_i$. This gives
\eqref{eq:bernstein-coefficients}. Each inner sum is a convex combination
of its coefficients on $[0,1]$, and so is the outer sum. This proves
the lower bound.
\end{proof}

We apply the Bernstein coefficient criterion to the two Gaussian
polynomials and to the numerator in \eqref{eq:middle-discriminant}.

\begin{lemma}\label{lem:arithmetic-certificate}
The polynomials \eqref{eq:mc-uniform-poly} and
\eqref{eq:gaussian-d-polynomial} satisfy
\begin{equation}\label{eq:gaussian-polynomial-positive}
 p_U(r)>\frac{119}{160}\quad(0\le r\le3/5),\qquad
 p_C(q)>\frac{399}{4000000}\quad(0\le q\le13/50).
\end{equation}
The integer polynomial $N_{\mathrm f}$ defined by
\eqref{eq:middle-discriminant} is strictly positive throughout
$\mathcal D$.
\end{lemma}

\begin{proof}
We expand each Gaussian polynomial to first order at a chosen
center and bound the remainder divided by the squared displacement
from below. Combining this bound with the value and derivative at
the center will prove positivity by completing the square.
We use centers $4/9$ and $1/7$, for which
\eqref{eq:gaussian-center-checks} gives the required value and
derivative bounds. Define the remainder quotients by
\begin{align*}
 Q_U(r)&=\frac{p_U(r)-p_U(4/9)-p_U'(4/9)(r-4/9)}{(r-4/9)^2},\\
 Q_C(q)&=\frac{p_C(q)-p_C(1/7)-p_C'(1/7)(q-1/7)}{(q-1/7)^2}.
\end{align*}
These are polynomials, including at the apparent singularities. Put
\begin{equation}\label{eq:gaussian-bernstein-polynomials}
 \mathscr B_U(z)=Q_U(3z/5),\qquad \mathscr B_C(z)=Q_C(13z/50).
\end{equation}
Their degrees are $11$ and $13$, respectively. For the polynomial used
in the intermediate range of $\mu$, put
\begin{equation}\label{eq:middle-bernstein-polynomial}
 \mu(z_1)=\frac{11}{20}+\frac7{40}z_1,\qquad
 H_{\mathrm f}(z_1,z_2)=10^{-21}N_{\mathrm f}
                 \bigl(\mu(z_1),\mu(z_1)z_2/2\bigr).
\end{equation}
It has bidegree $(12,3)$, and the substitution parametrizes all of
$\mathcal D$ as $(z_1,z_2)$ ranges over $[0,1]^2$.

Let $b_i^U,b_i^C$ be the Bernstein coefficients of
$\mathscr B_U,\mathscr B_C$ in degrees $11$ and $13$, respectively.
For $H_{\mathrm f}$, use degree $12$ in $z_1$ and choose the degree
$n_i\ge3$ in $z_2$ for each row so that every resulting coefficient
is positive. The separate row degrees are permitted by
\eqref{eq:bernstein-coefficients}.
Write $\beta_{ij}$ for the resulting coefficients of $H_{\mathrm f}$.
In \zcref{tab:bernstein-certificate}, $U_i,C_i,b_i$ denote the
integer lower bounds in the columns headed $b_i^U$,
$1000b_i^C$, and $100\beta_{ij}$, respectively.
Thus, wherever the corresponding table entry is present,
\begin{equation}\label{eq:bernstein-integer-bounds}
 b_i^U\ge U_i,\qquad 1000b_i^C\ge C_i,\qquad
 100\beta_{ij}\ge b_i\quad(0\le j\le n_i).
\end{equation}
A dash indicates an index above the degree in the corresponding expansion.
To reproduce the table, first expand the rational definitions of $p_U,p_C$
and $N_{\mathrm f}$, cancel their stated polynomial factors, make the
substitutions \eqref{eq:gaussian-bernstein-polynomials} and
\eqref{eq:middle-bernstein-polynomial}, and apply the finite sums in
\eqref{eq:bernstein-coefficients}. For example, the first univariate coefficient is
\[
 b_0^U=\mathscr B_U(0)=\frac{2358459066731126372}{1362025156640625}
 =1731+\frac{793520586204497}{1362025156640625}>1731.
\]
For $H_{\mathrm f}(z_1,z_2)=\sum a_{kl}z_1^kz_2^l$, its first row uses degree
$n_0=3$ and has coefficients
\[
 \beta_{0j}=a_{00}+\frac j3a_{01}
       +\frac{\binom j2}{3}a_{02}+\binom j3a_{03},
 \qquad 0\le j\le3.
\]
This gives the four comparisons $100\beta_{0j}\ge4355$ in the first
row of the last column. The other rows use the same finite formula with
the displayed $n_i$. The last column bounds every coefficient in its row.
The table records $26$ coefficient bounds for the Gaussian estimates and
$\sum_{i=0}^{12}(n_i+1)=63$ for intermediate $\mu$.

\begin{table}[htbp]
\centering
\begin{tabular}{r r r r r}
\toprule
&\multicolumn{2}{c}{Lower bounds}&&Common lower bound\\
$i$&$b_i^U$&$1000b_i^C$&Degree $n_i$&$100\beta_{ij}$, $0\le j\le n_i$\\
\midrule
0&1731&49&3&4355\\
1&1879&89&3&4993\\
2&2013&125&3&4752\\
3&2128&159&3&4288\\
4&2217&190&3&3753\\
5&2273&217&3&3162\\
6&2285&242&3&2526\\
7&2244&264&3&1848\\
8&2140&283&3&1126\\
9&1965&300&3&354\\
10&1721&313&5&317\\
11&1422&324&6&218\\
12&--&332&9&6\\
13&--&336&--&--\\
\bottomrule
\end{tabular}
\caption{Integer coefficient bounds in \eqref{eq:bernstein-integer-bounds}
and Bernstein degrees for the rows of $H_{\mathrm f}$.}
\label{tab:bernstein-certificate}
\end{table}

\zcref{lem:bernstein-coefficients} gives, on the respective unit
intervals and square,
\begin{equation}\label{eq:three-bernstein-bounds}
 \mathscr B_U\ge1422>1000,\qquad \mathscr B_C\ge\frac{49}{1000}>\frac1{25},\qquad
 H_{\mathrm f}\ge\frac3{50}>0.
\end{equation}
The last inequality proves $N_{\mathrm f}>0$ throughout $\mathcal D$.
For $p_U$ and $p_C$, direct substitution in
\eqref{eq:mc-uniform-poly} and \eqref{eq:gaussian-d-polynomial} gives
\begin{equation}\label{eq:gaussian-center-checks}
 \begin{gathered}
 p_U(4/9)>\frac34,\qquad |p_U'(4/9)|<5,\\
 p_C(1/7)>\frac1{10000},\qquad |p_C'(1/7)|<\frac1{5000}.
 \end{gathered}
\end{equation}
Completing the square in each Taylor expansion, using
\eqref{eq:three-bernstein-bounds}, yields
\[
 \begin{aligned}
 p_U(r)&>\frac34-\frac{5^2}{4000}=\frac{119}{160}>0,\\
 p_C(q)&>\frac1{10000}-\frac{(1/5000)^2}{4/25}
        =\frac{399}{4000000}>0.
 \end{aligned}
\]
These estimates hold on $[0,3/5]$ and $[0,13/50]$, respectively,
and complete the proof.
\end{proof}

\section{The zero-field case}\label{app:zero-field}

The zero-field support description is due to \cite{Lopatto2026}.
We give a proof using the estimates above and record their consequences
for the endpoint density and atom. In particular, the bound for the
upper atom below improves the estimate $a_\beta>1/3$ of
\cite[Theorem~1.2 and Proposition~2.3]{Chen2026EndpointMass} to
$a_\beta>\sqrt2-1$.

\begin{theorem}\label{thm:zero-field}
For $0<\beta\le1$, $\nu_{\beta,0}=\delta_0$. For every $\beta>1$, there exist $q_\beta\in(0,1)$ and $a_\beta>\sqrt2-1$ such that
\[
 \supp\nu_{\beta,0}=[0,q_\beta],\qquad
 \nu_{\beta,0}(\dd q)=f_\beta(q)\mathbf1_{(0,q_\beta)}\dd q
                         +a_\beta\delta_{q_\beta}(\dd q).
\]
The density $f_\beta$ is smooth and strictly positive on $[0,q_\beta]$,
where every derivative has a continuous one-sided extension from the
interior. In particular, $\nu_{\beta,0}(\{0\})=0$ when $\beta>1$.
\end{theorem}

\begin{proof}
\emph{Step 1: The moment inequality at positive times.}
The PDE estimates in \zcref{sec:profiles} apply unchanged at $h=0$. 
\zcref{lem:lower-density} also applies at $h=0$: then $k_0=2K_t$
and, in \eqref{eq:lower-ratio-pde},
$\mathfrak d(\sigma,r)=1-2r/\sigma$ and $\mathfrak d_{rr}=0$. In the moment argument, $\theta=0$ and $d>0$
satisfy \eqref{eq:moment-parameter-bounds} automatically. The weighted density
has Gaussian decay and a finite positive normalizing constant. The
probability measure defining $\widehat\E$ in \zcref{sec:weighted-moments}
is even when $h=0$, so $\widehat\E u=0$.
Under the one-sided conditions on $m_0$, the identity $F=2m_0G$
excludes $m_0=0$ at a positive time: the density of $Y_t$ is
strictly positive and $C_x(t,\cdot)$ is not identically zero, so
$F(t)>0$. For $m_0>0$, the normalized functions and law satisfy
the standing assumptions of \zcref{sec:weighted-moments} with
$\theta=0$. The range $\mu\ge29/40$ is excluded by
\zcref{prop:large-mu-exclusion}; the other two ranges give $O>0$
by \zcref{prop:small-positive,prop:middle-positive}.
The scaling identity \eqref{eq:weighted-moment-scaling} therefore gives
\begin{equation}\label{eq:zero-derivative-moments}
 F(t)-2m_0G(t)=0\quad\Longrightarrow\quad J(t;m_0)>0
 \qquad(0<t<L)
\end{equation}
for every $m_0$ satisfying the one-sided inequalities in
\zcref{thm:derivative-moments}.

\medskip
\emph{Step 2: The origin and the critical temperature.}
Zero belongs to the support by
\cite[Theorem~1(i)]{AuffingerChen2015Properties}. We reproduce its
monotonicity argument. The classical
critical-temperature criterion is discussed in
\cite{AizenmanLebowitzRuelle1987,Toninelli2002};
\cite[Remark~2]{Chen2017Variational} gives the related short variational
proof.
Put $L=\beta^2$ and $\nu=\nu_{\beta,0}$. By parity $Q(0)=0$.
\zcref{lem:variational-support} excludes $1$ from the support and
gives \eqref{eq:support-self-consistency} at every positive support point.
If the least support point were $a>0$, then $m=0$ before ${t_-}=La$,
so $S'=F>0$ there and
\[
 \frac {t_-} L=Q({t_-})=\int_0^{t_-} S(t)\dd t<{t_-} S({t_-})\le\frac {t_-} L,
\]
a contradiction. Here $F(t)>0$ for $t>0$, since $C_x(t,\cdot)$ is not
identically zero and the diffusion law has positive density. Hence $0\in\supp\nu$. The identity $Q(0)=0$ noted above and the formula
for the first derivative of the variational potential give
\[
 \mathscr V'(0)=\frac L2\bigl(0-Q(0)\bigr)=0.
\]
Since $0$ is a one-sided minimum of $\mathscr V$, one also has
$\mathscr V''(0)\ge0$, and therefore
\begin{equation}\label{eq:zero-stability}
 S(0)=C(0,0)^2\le L^{-1}.
\end{equation} 
If $L\le1$, then $S(t)<1$ for $t>0$: use $C\le1-B^2$, strict
monotonicity of $B$, and the positive diffusion density. Therefore
$Q(t)<t\le t/L$, the variational potential strictly increases, and
$\nu=\delta_0$. Conversely, $\nu=\delta_0$ gives $m=1$ for $t>0$
and $C(0,0)=1$, contradicting \eqref{eq:zero-stability} when $L>1$.

\medskip
\emph{Step 3: Connected support.}
Assume $L>1$. Step~1 of the proof of \zcref{thm:main} excludes
gaps with positive left endpoint, using \eqref{eq:zero-derivative-moments}.
For a gap $(0,b)$, put ${t_+}=Lb$ and $m_0=m(0)$. Then $m=m_0$ on
$(0,t_+)$. Now $N(0)=0$, $N({t_+})\ge0$, and 
\begin{equation}\label{eq:zero-gap-integral}
 \int_0^{t_+}\frac{N(t)}{t^2}\dd t=S(0)-\frac1L\le0.
\end{equation}
Indeed, $Q(0)=0$, $Q'=S$, and Lipschitz continuity gives
\[
 N(t)=\int_0^t(S(r)-S(t))\dd r=O(t^2),\qquad
 \frac{Q(t)}t\longrightarrow S(0).
\]
Integrating $N(t)/t^2=-(Q(t)/t)'$ from $\varepsilon$ to $t_+$ and
letting $\varepsilon\downarrow0$ proves \eqref{eq:zero-gap-integral}.
If $N\not\equiv0$, \eqref{eq:zero-gap-integral} forces a negative
interior minimum. At that minimum $N'=-tQ''=0$ and
$N''=-Q''-tJ(t;m_0)\ge0$, contradicting
\eqref{eq:zero-derivative-moments}. If $N\equiv0$, then
$Q''=J(t;m_0)=0$ gives the same contradiction. Thus the
support is $[0,q_\beta]$ with $0<q_\beta<1$.

\emph{Step 4: Density and endpoint masses.}
Set $t_+=Lq_\beta$. The quotient argument from Step~2 of the proof of
\zcref{thm:main}, beginning with \eqref{eq:cdf-quotient}, applies here:
$Q=t/L$ on $[0,t_+]$, and, by right-continuity of $m$ and continuity of
$F/(2G)$,
\[
 m(t)=\frac{F(t)}{2G(t)}
 \qquad\text{for }0\le t<t_+.
\] 
For differentiation at $t_+$, we use $F/(2G)$ as the continuous
extension of the interior values of $m$, with value $m(t_+-)$; the original
right-continuous distribution function has value $m(t_+)=1$. 
The quotient is Lipschitz, so the measure has no singular part in the
interior of its support. The induction using polynomial-moment
identities in Step~2 of that proof gives smoothness with one-sided
derivatives. At zero,
$F(0)=C_x(0,0)^2=0$, so right-continuity gives $m(0)=0$ and there is
no atom there. For $0<t\le {t_+}$, \eqref{eq:zero-derivative-moments} and $2Gm'=J$
give $m'>0$. At zero, $m(0)=0$ and $C_x(0,0)=0$ imply, from the
continuous one-sided form of \eqref{eq:ito-F}, that
$F'(0)=C_{xx}(0,0)^2$. Differentiating the
quotient at zero therefore gives
\begin{equation}\label{eq:zero-origin-density}
 m'(0)=\frac{C_{xx}(0,0)^2}{2C(0,0)^3}>0,
\end{equation}
because the normalization parameter in \zcref{lem:log-derivative} is positive.
Thus $f_\beta(q)=Lm'(Lq)$ is smooth and strictly positive on the closed
interval. In particular, $C(0,0)=L^{-1/2}$ since $S=1/L$.

At the upper endpoint use the moment values in
\eqref{eq:terminal-moments}, writing $m=m({t_+}-)\in(0,1]$.
Apply \zcref{lem:general-mean-bound} with $\theta=0$ and
$\Sigma\ge d>0$ to obtain $y\ge a_0+\Sigma/2>a_0$, hence
\[
 \frac{1-m}{m}>1-\frac m2,\qquad m^2-4m+2>0.
\]
Since $0<m\le1$, this gives $m<2-\sqrt2$ and
$a_\beta=1-m>\sqrt2-1$. The absence of mass outside the support and
of an atom at zero gives the stated decomposition.
\end{proof}

\bibliographystyle{plainnat}
\bibliography{positive_field.bib}
\end{document}